\documentclass[11pt]{amsart}

\usepackage{fullpage,url,amssymb,enumerate,colonequals}
\usepackage{url,amssymb,enumerate,colonequals}
\usepackage{multirow}
\usepackage{mathrsfs} 
\usepackage[section]{placeins}
\usepackage{MnSymbol}
\usepackage{extarrows}
\usepackage{lscape}
\usepackage[all,cmtip]{xy}
\usepackage[OT2,T1]{fontenc}
\usepackage{color}

\usepackage[
        colorlinks, citecolor=darkgreen,
        backref,
        pdfauthor={Elvira Lupoian},
]{hyperref}

\usepackage{comment}
\usepackage{multirow}
\usepackage[table]{xcolor}
\usepackage{tabularray}
\usepackage{hyperref}
\hypersetup{
    colorlinks=true,
    linkcolor=darkgreen,
    filecolor=magenta,      
    urlcolor=darkgreen,
    pdftitle={Overleaf Example},
    pdfpagemode=FullScreen,
    }

\usepackage{amsmath}
\usepackage{amsthm}
\usepackage{amssymb}
\usepackage{tikz-cd}
\usepackage{tikz}
\usepackage{tabularx}
\usepackage{hyperref}

\newcommand{\Q}{\mathbb{Q}}
\newcommand{\R}{\mathbb{R}}
\newcommand{\Z}{\mathbb{Z}}

\newcommand{\calC}{\mathcal{C}}

\newcommand{\cF}{\mathcal{F}}

\DeclareMathOperator{\divv}{div}
\DeclareMathOperator{\Div}{Div}

\DeclareMathOperator{\Gal}{Gal}

\DeclareMathOperator{\Divv}{Div}
\DeclareMathOperator{\Diag}{Diag}
\DeclareMathOperator{\Bl}{Bl}
\DeclareMathOperator{\Circ}{Circ}
\DeclareMathOperator{\Circa}{Circalt}

\newcommand{\tors}{{\operatorname{tors}}}

\newcommand{\SL}{\operatorname{SL}}

\newcommand{\isom}{\simeq}

\newcommand{\floor}[1]{\left\lfloor #1 \right\rfloor}

\numberwithin{equation}{section}

\newtheorem{thm}{Theorem}
\newtheorem{lem}[thm]{Lemma}
\newtheorem{prop}[thm]{Proposition}
\newtheorem{cor}[thm]{Corollary}

\theoremstyle{definition}

\newtheorem*{Ack}{Acknowledgements}

\theoremstyle{remark}

\theoremstyle{definition}
\newtheorem{definition}[equation]{Definition}

\theoremstyle{remark}
\newtheorem{remark}[equation]{Remark}

\newenvironment{psmallmatrix}
  {\left(\begin{smallmatrix}}
  {\end{smallmatrix}\right)}

\definecolor{darkgreen}{rgb}{0,0.5,0}

\DeclareRobustCommand{\SkipTocEntry}[5]{}

\begin{document}
\title{Modular units on  \( X_{1}( p) \) and quotients of the cuspidal group}

\begin{abstract}
Modular units are functions on modular curves whose divisors are supported on the cusps. They form a free abelian group of rank at most one less than the number of cusps. In this paper we study the group of modular units on $X_{1}( p )$, with prime level $p \ge 5$. We give an explicit basis for this group and study certain rational subgroups of it. We use the basis to numerically investigate the structure of the cuspidal group of $X_{1}( p)$ and its rational subgroup. In the later stages of this paper we use our basis to determine a specific large quotient of the cuspidal group.
\end{abstract}

\author{Elvira Lupoian}

\address{Department of Mathematics, University College London, London, WC1H 0AY, UK}

\email{e.lupoian@ucl.ac.uk}
\date{\today}
\thanks{The author is supported by the EPSRC grant EP/W524335/1}
 \maketitle
\section{Introduction}
Let $\Gamma < \SL_{2}(\Z)$ be a congruence subgroup and $X_{\Gamma}$ the complete modular curve associated to  $\Gamma$. The subgroup of the Jacobian $J_{\Gamma}$ of $X_{\Gamma}$ generated by linear equivalence classes of differences of cusps is denote by $\mathcal{C}_{\Gamma}$ and it is known as \textit{cuspidal group} of $X_{\Gamma}$. This group is finite by the classical theorems of Manin \cite{Manin} and Drinfeld \cite{Drinfeld}, and its order $h_{\Gamma}$ is often called the \textit{cuspidal class number}. The proof of Manin-Drinfeld is ineffective and gives no method for computing $h_{\Gamma}$. When $\Gamma = \Gamma_{0}(p)$, with $p \ge 5$ prime, $X_{\Gamma} = X_{0}(p)$ has two cusps $c_0$, $c_\infty$, both rational points on the curve, and Ogg \cite{Ogg} showed that their difference defines an element of order $\frac{p-1}{\gcd(12, p-1)}$ on the Jacobian. Mazur \cite{mazur1977modular} showed that this element generates the entire rational torsion subgroup of the Jacobian $J_{0}(p)$. 

In general, computing $\mathcal{C}_{\Gamma}$ or $h_{\Gamma}$ remains a difficult problem. The cuspidal class number or certain factors of it has been determined in some cases. Kubert and Lang \cite{KL} computed $h_{\Gamma(N)}$ when $N$ is coprime to $6$.  In a series of papers Takagi \cite{takagi1997cuspidal}, \cite{takagi1992cuspidal} \cite{takagi1993cuspidal}, \cite{takagi1995cuspidal} determined $h_{\Gamma_{1}(N)}$ and $h_{\Gamma_{0}(N)}$ for various levels $N$. Yu \cite{yu1980cuspidal} completely determined the order of the subgroup of the cuspidal subgroup of $X_{1}(N)$ generated by $\infty$-cusps. The structure of the cuspidal group is less studied. Most result treat the subgroup generated by classes of rational divisors, and almost all results concern the modular curve $X_{0}(N)$, see \cite{chua1997rational}, \cite{lorenzini1995torsion}, \cite{ling1997q}, \cite{rouse2015spaces}, \cite{takagi1997cuspidal}. The most complete results is that of Yoo \cite{yoo2023rational}, which completely determines the structure of this rational subgroup of the cuspidal group. The structure of the $p$-primary of a specific subgroup of the cuspidal group of $X_{1}(p^{n})$ for $n\ge 2$ was studied by Yang and Yu \cite{yang2010structure}. There are  various results comparing the cuspidal group to the rational torsion subgroup of the Jacobian, see \cite{lorenzini1995torsion}, \cite{ohta2013eisenstein}, \cite{masami2014eisenstein}, \cite{yoo2023rational2}.

In this paper we focus on $\Gamma = \Gamma_{1}(p)$, with $p \ge 5$ prime. In this setting, $h_{1}(p) = h_{\Gamma_{1}(p)}$  is completely determined due to Takagi \cite{takagi1992cuspidal} and Yu \cite{yu1980cuspidal}. Moreover, Ohta \cite{ohta2013eisenstein} proved that the rational torsion subgroup of the Jacobian is generated by cusps up to $2-$torsion. However, the structure of the cuspidal group remains a mystery. Computationally this is a challenging problem also. The genus of $X_{1}(p)$ grows extremely fast and thus computing a model of the curve and the cuspidal group explicitly is impractical.  An added difficulty is the fact the number of cusps grows with the level. Motivated by this, we want to find a practical and efficient way of computing the cuspidal group of  $X_{1}(p)$. 

One way of achieving this is by computing a basis for the group of \textit{modular units} for $\Gamma_{1}(p)$, that is, the functions on $X_{1}(p)$ whose divisors are entirely supported on cusps. Given such a basis, the problem of explicitly computing the cuspidal group is reduced to a simple linear algebra problem, since then one would simply need to compute the Smith normal form of the matrix representing the divisors of the basis.

The main result of this paper constructs an explicit basis for the group of modular units. Our strategy is to use Siegel functions to construct modular units on the curve and then use the known order of the cuspidal group to verify that a certain set of modular units is a basis. This strategy was also used by Yang in \cite{yang2007modular} to study the subgroup of modular units whose divisors are supported on the cusps lying over $c_{\infty} \in X_{0}(p)$ only. We reconstruct Yang's basis using our functions.

Modular units have been studied for other families of modular curves. Kubert and Land \cite{KL} used Siegel functions to construct modular units on $X(N)$. The modular units on $X_{1}(N)$ considered in the previously mentioned works of Yu and Takagi are also products of Siegel functions. For $X_{0}(N)$, modular units are typically products of eta functions, see for instance \cite{Ogg}, \cite{newman1957construction}, \cite{newman1959construction}, \cite{takagi1997cuspidal}.

More generally, modular units have wider arithmetic applications in the study of modular curves. For instance, they are used in the work of Bilu and Parent \cite{bilu2011serre}, \cite{bilu2011runge} to study integral points on modular curves. They are used for gonality computations in \cite{derickx2014gonality}. Ogg's modular unit is used by Gross \cite{gross1984heegner} to prove that certain Heegner points are non-torsion.  

\subsection{Set up} 
In this paper we study the group of modular units on $X_{1}(p)$, where $p \ge 5$ is prime.
\begin{definition}
    A \textit{modular unit} for a congruence subgroup $\Gamma \subseteq \SL_{2}( \Z)$ is a meromorphic function on the upper half plane $\mathbb{H}$, with no zeros or poles in $\mathbb{H}$ and which is invariant under the action of $\Gamma$. 
\end{definition}
It follows that any modular unit for $\Gamma$ is a function on the associated complete modular curve $X_{\Gamma}$, whose divisor is supported on cusps only. It is for this reasons that we refer to a modular unit for $\Gamma$ as a \textit{modular unit on} $X_{\Gamma}$, and vice versa. Throughout this paper we use these two notions interchangeably. 

The modular units that we construct are products of \textit{Siegel functions}, see Section \ref{siegelfuns} for details. Associated to any $\mathbf{a} \in \Q^{2} \setminus \Z^{2}$ there is a meromorphic function $g_{\mathbf{a}} : \mathbb{H} \longrightarrow \mathbb{C}$ which has no zeros or poles in $\mathbb{H}$. The work of Kubert and Lang (see Theorem \ref{KLthm}) completely describes the modular units of $\Gamma(p)$ in terms of Siegel functions. Our aim is to understand when such functions are invariant under the action of $\Gamma_{1}(p) \supseteq \Gamma(p) $. This leads us to the following natural construction, as detailed in Section \ref{constructions}.

For the rest of this paper we fix a prime number $p \ge5$ and an integer $\alpha $ whose reduction modulo $p$ has order $p-1$. All constructions in this paper depend on this choice of $\alpha$.
Let $n = \frac{p-1}{2}$ and for $0 \le i \le n -1$ define:
\begin{center}
 $E_{i} ( \tau ) = g_{\left( \frac{\alpha^{i}}{p}, 0 \right)} ( \tau ) \displaystyle \prod_{j=0}^{p-2} g_{\left( \frac{\alpha^{i}}{p}, \frac{\alpha^{j}}{p} \right)}( \tau ) \ \ \ \ \text{and} \ \ \ \ F_{i} ( \tau ) = g_{\left( 0, \frac{\alpha^{i}}{p} \right)} ( \tau ).$
 \end{center}

In Section \ref{constructions} we give necessary and sufficient conditions for products of such functions to be modular units on $X_{1}(p)$ and we study the divisors of such functions. Crucially, we are also able to use the fixed $\alpha$ to parametrise the cusps of $X_{1}(p)$. More precisely, the cusps of $X_{1}(p)$ fall into a natural partition:
\begin{itemize}
    \item the cusps $P_{i}$ represented by $\frac{\alpha^{i}}{p}$ with $ 0 \le i \le n-1$;
    \item  the cusps $Q_{j}$ represented by $\frac{p}{\alpha^{j}}$ with $0 \le j \le n-1$.
\end{itemize}
This choice of representatives is well suited to computing the divisor of any modular unit on $X_{1}(p)$ constructed as a product of the above functions, see Proposition \ref{ordsvansX1}.

As detailed in Section \ref{modularcurves}, we work with the canonical model of $X_{1}(p)$ over $\Q$ in which the cusps $P_{0}, \ldots, P_{n-1}$ are defined over $\Q$ and $Q_{0}, \ldots, Q_{n-1}$ are defined over $\Q(\zeta_{p})^{+}$, the maximal real subfield of the cyclotomic field $\Q( \zeta_{p})$, and form one orbit under the induced Galois action.

\subsection{Main Results}

Let $\beta \in \{ \pm 1, \pm 5 \}$ with $p \equiv  \beta \mod{12} $. For $0 \le i \le n-2$ define
\begin{align}
     G_{i} ( \tau  ) = E_{i} ( \tau  ) E_{n-1} ( \tau  )^{-\alpha^{2i+2}}F_{n-1} ( \tau  )^{p (\alpha^{2i+2} -1)} \ \  \  \text{and} \ \ \ H_{i} ( \tau  ) = F_{i}( \tau  ) F_{n-1}( \tau  )^{ -\alpha^{2i+2} + p \beta (\alpha^{2i+2} -1)},
\end{align}
and further define 
\begin{align}
G_{n-1} ( \tau  ) = E_{n-1} ( \tau  )^{p} F_{n-1} ( \tau  )^{-\beta p} \ \ \text{and} \ \ 
H_ {n-1} ( \tau  ) = F_{n-1}( \tau  )^{12p}.
\end{align}
These are all modular units for $\Gamma_{1}(p)$ by the criteria proved in Section \ref{constructions}. In Section \ref{basissec} we use the known cuspidal class number of $X_{1}(p)$ (see Theorem \ref{cuspidalclassnumberX1p}) to prove the following. 
\begin{thm} \label{basisx1p2}
The $p-2$ functions $G_{0}, \ldots, G_{n-1}, H_{1}, \ldots, H_{n-1}$ form a basis for the group of modular units on $X_{1}(p)$ modulo constants.
\end{thm}

This basis allows us to compute the cuspidal group efficiently. We write $\Div_{c}(p)$ for the group of degree $0$ divisors on $X_{1}(p)$ supported on cusps. The cuspidal group is simply the quotient
\begin{center}
    $\mathcal{C}_{1}(p) := \Div_{c}(p) / \langle \divv(G_{i}), \divv(G_{j} :0 \le i,j \le n-1 \rangle$.
    \end{center}
Therefore, computing these groups abstractly reduces to determining the Smith normal form of the matrix representing the divisors of the explicit bases above. We discuss some examples in Section \ref{examples1}.

There are two natural subgroups of the group of  modular units. Firstly, we consider the subgroup of modular units whose divisors are supported on $P_{0}, \ldots, P_{n-1}$. Let $\gamma = \alpha^{p-2}$ and for $1 \le i \le n-2$ define
\begin{align*}
 I_{i} ( \tau  ) & = E_{i-1} ( \tau  ) E_{i} ( \tau  )^{ -\gamma^{2} -1}E_{i+1} ( \tau  )^{\gamma^{2}}
 \end{align*}
and $I_{n-1} ( \tau  ) = \left(E_{n-2} ( \tau  )E_{n-1} ( \tau  )^{-1} \right)^{p}$. These functions are all modular units on $X_{1}(p)$ by Proposition \ref{X1pcongs} and one uses Proposition \ref{ordsvansX1} to verify that their order of vanishing at any $Q_{i}$ is zero. 
In Section \ref{subgroups} we prove the following. 
\begin{thm} \label{basis2}
The functions $I_{1}, \ldots, I_{n-1}$ are a basis for the group of modular units on $X_{1}(p)$ whose divisors are supported on $P_{0}, \ldots, P_{n-1}$, modulo constants.
\end{thm}

The second subgroup we consider is that consisting of modular units whose divisors are fixed by the action of Galois. By our choice of model, this group contains the previous subgroup and we extend our basis by considering the following function. If $p \equiv 1 \mod{4}$, we define
\begin{equation} \label{Cs}
    I_{n} \left( \tau \right) = 
        \left( E_{0} \left( \tau \right) E_{\frac{n}{2}} \left( \tau \right) \right)^{6}.
\end{equation}
When $p \ \equiv \ 3 \mod{4}$, let $m$ be the unique integer $0 \le m <n$ satisfying 
\begin{equation} 
   \alpha^{2m} \equiv  \begin{cases}
       -2  \mod{p} & \text{if} \ \ p  \equiv 3 \mod{8} \\
         2 \mod{p} & \text{if} \ \  p  \equiv 7 \mod{8}    \end{cases}
\end{equation}
and define 
\begin{equation} 
   I_{n} \left( \tau \right) =  \begin{cases}
       E_{0} \left( \tau \right)^{8} E_{m} \left(\tau \right)^{4}  \ & \text{if} \  p  \equiv 3 \mod{8} \\
         E_{0} \left( \tau\right)^{24} E_{m} \left( \tau \right)^{-12} \ & \text{if} \  p  \equiv 7 \mod{8}.
         \end{cases}
\end{equation}

\begin{thm} \label{basis3}
The functions $I_{1}, \ldots, I_{n-1}, I_{n}$ form a basis for the group of modular units on $X_{1}(p)$ with rational divisors, modulo constants. 
\end{thm}

Furthermore, in Section \ref{subgroups} we use the above bases to prove that the rational cuspidal group is generated by equivalence classes of divisors supported on $P_{0}, \ldots, P_{n-1}$. This gives us a practical method for computing the rational cuspidal group, and we discuss our computations in Section \ref{examples1}. 

Moreover, these bases allow us to find `good generators' for the cuspidal group and its rational subgroup.  We discuss this further in Section \ref{examples1}. We found that in all of our calculations, the linear equivalence classes of $P_{0} -Q_{0}$ and $P_{0} - P_{n-1}$ are prominent in our representatives. They appear to generate a large factor of the groups themselves, and this motivates us to compute their orders, for which we find explicit formulae in Section \ref{ordersofcusps}.

\subsection{Overview} This paper is organised as follows. In Section \ref{modularcurves} we review some basic facts about modular curves and their cusps. Section \ref{siegelfuns} gives the necessary background on Siegel functions required for our constructions. In Section \ref{constructions} we prove a necessary and sufficient criteria for certain products on Siegel functions to be modular units on $X_{1}(p)$. In Section \ref{basissec} we prove Theorem \ref{basisx1p2} using the results of Section \ref{constructions} and the known cuspidal class number. In Section \ref{subgroups} we prove Theorems \ref{basis2}  and \ref{basis3}. In Section \ref{examples1} we use our explicit bases to explicitly compute the cuspidal group and its rational subgroup for $11 \le p \le 997$, and analyse their structures. In the final section we explicitly determine the orders of two specific elements of the cuspidal group, which appear to play a crucial role in its structure. 
\begin{Ack} The author sincerely thanks Samir Siksek and Damiano Testa for many helpful conversations. The author is supported by the EPSRC Doctoral Prize fellowship EP/W524335/1.
\end{Ack} 

\subsection*{Notation}
\begin{itemize}
    \item[] $p \ge 5$ : prime number 
    \item[] $n := \frac{1}{2}(p-1)$
\item[] $\alpha$ : a fixed integer whose reduction modulo $p$ generates $(\Z / p \Z)^{*} $
\item[]$\beta$ : the unique $\beta \in \{ \pm 1, \pm 5 \}$ such that $p \equiv \beta \mod{12}$
\item[] $\omega$ : a primitive $n$th root of unity 
\item[] $P_{i}$ : the cusp of $X_{1}(p)$ represented by $\frac{\alpha^{i}}{p}$, with $0 \le i \le n-1$
\item[] $Q_{i}$ : the cusp of $X_{1}(p)$ represented by $\frac{p}{\alpha^{i}}$, with $0 \le i \le n-1$
\item[] $g_{\mathbf{a}}(\tau)$ : the Siegel function defined by $\mathbf{a} \in \Q^{2}$
\item[] $\{ x \}$ : the fractional part of $x \in \R$
\item[] $B_{2}(x) = x^{2} - x + \frac{1}{6}$ : the second Bernoulli polynomial
\item[] $a_{i} = \frac{p}{2} B_{2} ( \{ \frac{\alpha^{i}}{p} \})$ for $ 0 \le i \le n-1$ 
\item[] $B_{2, \chi} = p \sum_{a=1}^{p} \chi(a) B_{2}(\frac{a}{p})$ where $\chi$ is a Dirichlet character modulo $p$
\item[] $\Div_{c}(p)$ : the group of degree $0$ divisors on $X_{1}(p)$, supported on cusps 
\item[] $\Div_{c}^{\Q}(p)$ : the group of degree $0$ divisors on $X_{1}(p)$, supported on cusps and fixed by the Galois action
\item[] $\Div_{c}^{\infty}(p)$ : the group of degree $0$ divisors on $X_{1}(p)$, supported on $P_{1}, \ldots, P_{n}$
\item[] $\mathcal{F}(p)$ : the group of modular units on $X_{1}(p)$
\item[] $\mathcal{F}^{\Q}(p)$ : the group of modular units on $X_{1}(p)$ whose divisors are in 
$\Div_{c}^{\Q}(p)$ 
\item[] $\mathcal{F}^{\infty}(p)$ : the group of modular units on $X_{1}(p)$ whose divisors are in $\Div_{c}^{\infty}(p)$ 
\item[] $\mathcal{C}_{1}(p) := \Div_{c}(p) / \langle \divv(f) : f \in \mathcal{F}(p) \rangle$
\item[] $\mathcal{C}_{1}^{\Q}(p) := \Div_{c}^{\Q}(p) / \langle \divv(f) : f \in \mathcal{F}^{\Q}(p) \rangle$
\item[] $\mathcal{C}_{1}^{\infty}(p) := \Div_{c}^{\infty}(p) / \langle \divv(f) : f \in \mathcal{F}^{\infty}(p) \rangle$
\item[] $h_{1}(p) = \vert \mathcal{C}_{1}(p) \vert$
\item[] $h_{1}^{\infty} (p) = \vert \mathcal{C}_{1}^{\infty}(p)\vert$
\end{itemize}

\section{Modular Curves and their Cusps} \label{modularcurves}
In this section we briefly recall some key (and well-known) facts about modular curves; the reader should refer to \cite[Chapters 1-3]{diamond2005first} for details. 
Fix a positive integer $N$ and let $X\left( N\right) / \Q $, $X_{1}\left( N\right) / \Q $ and $X_{0}\left( N\right) / \Q $ denote the  Shimura model over $\Q$ of the classical modular curves associated to the congruence subgroups $\Gamma \left( N \right)$, $\Gamma_{1} \left( N\right)$ and $\Gamma_{0}\left( N\right)$ respectively, which are classically defined as:
\begin{align*}
       & \Gamma \left( N\right) = \{ M \in \text{SL}_{2}\left( \Z \right) \ \vert \ M \equiv \begin{psmallmatrix}
        1 & 0 \\
        0 & 1
    \end{psmallmatrix} \mod{N} \}; \\
     & \Gamma_{1}\left( N \right) = \{ M \in \text{SL}_{2}\left( \Z \right) \ \vert \ M \equiv \begin{psmallmatrix}
        1 & * \\
        0 & 1
    \end{psmallmatrix} \mod{N} \}; \\   
       & \Gamma_{0}\left( N\right) = \{ M \in \text{SL}_{2}\left( \Z \right) \ \vert \ M \equiv \begin{psmallmatrix}
        * & * \\
        0 & *
    \end{psmallmatrix} \mod{N} \}.    \end{align*}

The inclusions  $ \Gamma \left( N\right) \subset \Gamma_{1} \left( N\right) \subset \Gamma_{0}\left( N\right)$ induce natural projection maps 
\begin{center}
   $ X ( N)  \xrightarrow{\pi_{1} } X_{1} ( N) \xrightarrow{\pi_{0} } X_{0}(N ) $ 
 \end{center}
which are non-constant, coverings of the corresponding modular curves. The non-cuspidal points of $X \left( N \right)$, $X_{1} \left( N\right)$ and $X_{0} \left( N\right)$ parametrise isomorphism classes of elliptic curves with a basis for its $N$-torsion subgroup, a torsion point of order $N$ and a cyclic subgroup of order $N$ respectively. The cusps of these curves correspond to orbits of $\mathbb{P}^{1}( \Q) := \Q \cup \{ i \infty \}$ under the action of the corresponding congruence subgroups, and in \cite{Ogg} Ogg gave an explicit description of these, which we briefly review here. From now on, we assume that $N \ge 11$, which ensures that $X_{1}( N )$ has positive genus.

The cusps of $X( N )$ can be regarded as pairs $\pm \begin{psmallmatrix}
    x \\ y 
\end{psmallmatrix}$, where $x,y \in \left( \Z / N \Z \right)$ are relatively prime. The group $G = \text{SL}_{2} \left( \Z  \right) / \Gamma\left( N\right) $ acts on the cusps, and thus the cusps of $X_{1} \left( N\right)$ and $X_{0} \left( N \right)$ are simply orbits of $\Gamma_{1} \left( N\right)/ \Gamma \left( N\right))$ and $\Gamma_{0} \left( N\right) / \Gamma \left( N\right)$.

As for rationality, we recall that in the canonical model of $X \left( N \right)$, the cusps are rational in $\Q \left( \zeta_{N} \right) = \Q \left( e^{2 \pi i /N} \right)$, see \cite{shimura1971introduction}. We choose the canonical model over $\Q$ in which the  Galois group $\Gal \left( \Q\left( \zeta_{N} \right) / \Q \right) \isom \left( \Z / N \Z \right)^{\times}$  acts as $\begin{psmallmatrix}
    1 & 0 \\ 0 & \sigma
\end{psmallmatrix}$. That is, for any cusp $P = \pm \begin{psmallmatrix}
    x \\ y 
\end{psmallmatrix}$ and $\sigma \in \left( \Z / p \Z \right)^{\times}$, representing the automorphism defined by $ \zeta_{p} \mapsto \zeta_{p}^{\sigma}$,  
\begin{center}
    $\sigma \left( P \right) = \pm \begin{psmallmatrix}
        x \\ \sigma y 
    \end{psmallmatrix}$.
\end{center}
The following results are proved in \cite[Proposition 1 and 2]{Ogg}.
\begin{prop} \label{cuspsX1p}
 Suppose $N  = 11$ or $N \ge 13$. The cusps of $X_{1}(N)$ are represented by pairs $ \pm \begin{psmallmatrix}
     x \\ y 
 \end{psmallmatrix}$ where $x,y$ are reduced modulo $N$, $\gcd(y,N) = d$ for some positive $d \vert N$ and $\gcd(x,y) = 1 $. The cusps with a fixed value of $\pm y$ are Galois conjugate. 
\end{prop}

\begin{prop}
Suppose that $N \ge 5$. The cusps of $X_{0}(N)$ are represented by $ \pm \begin{psmallmatrix}
     x \\ d
 \end{psmallmatrix}$ where $d \vert N$ is positive, and $x$ is relatively prime to $d$ and reduced modulo $\gcd(\frac{N}{d}, d )$.
 \end{prop}

\begin{remark}
For any cusp  represented by  $ \begin{psmallmatrix}
     x \\ y
 \end{psmallmatrix}$ we choose a representative $(x',y') \in \Z^{2}$ with $x',y'$ coprime, $y' \neq 0$ and $x \equiv x' \mod{N}$, $y \equiv y' \mod{N}$. Throughout this paper we refer to the rational number $\frac{x'}{y'}$ as a representative for the cusp. 
 
 \end{remark}

\section[Siegel]{Siegel Functions and Modular units on \texorpdfstring{$X(p)$}{k}} \label{siegelfuns}
In \cite{KL}, Kubert and Lang describe the group of modular units on $X ( N )$, when $N$ is a positive integer coprime to $6$. As in this reference, the modular units constructed in the proceeding sections are products of \textit{Siegel functions}. In this section we review essential properties of Siegel functions and state some key results due to Kubert and Lang. 

\subsection{Siegel Functions}
Let $\mathbf{a} = \left( a_{1}, a_{2} \right) \in \Q^{2}$ be such that $ \mathbf{a} \notin \Z^{2}$. The \textit{Siegel functions} associated to $\mathbf{a}$ is a function on the upper half plane defined by the infinite product presentation 
\begin{center}
    $g_{\mathbf{a}} ( \tau ) = - e^{2 \pi i a_{2} (a_{1} -1)/2} q_{\tau}^{B_{2}\left( a_{1} \right)/2} \left( 1 - q_{\tau} \right) \displaystyle \prod_{n=1}^{\infty} \left( 1 - q_{\tau}^{n} q_{z} \right) \left( 1 - q_{\tau}^{n}/q_{z} \right) $
\end{center}
where $z = a_{1} \tau + a_{2}$, $q_{\tau} = e^{2 \pi i \tau}$, $q_{z} = e^{2 \pi i z}$ and $B_{2} = x^{2} - x + \frac{1}{6}$ is the second Bernoulli polynomial. Equivalently, Siegel functions can be defined in terms of Klein forms and the Dedekind eta function, see \cite{KL} for details. The constructions of Section \ref{constructions} use the following elementary properties of $g_{\mathbf{a}}$.

\begin{prop} \label{trans2}
For $\mathbf{a} = \left( a_{1}, a_{2} \right) \in \Q^{2} \setminus  \Z^{2}$, $\mathbf{b} = \left( b_{1}, b_{2} \right) \in \Z^{2}$  and $\gamma = \begin{psmallmatrix}
        a & b \\ c & d \end{psmallmatrix} \in SL_{2}\left( \Z \right)$
\begin{itemize}
    \item[i)] $ g_{-\mathbf{a}} ( \tau )  =  - g_{\mathbf{a}} ( \tau )$;
    \item[ii)] $ g_{\mathbf{a} + \mathbf{b}} ( \tau )  = \epsilon \left( \mathbf{a}, \mathbf{b} \right) g_{\mathbf{a}}( \tau )$ where $\epsilon \left( \mathbf{a}, \mathbf{b} \right) = \left( -1 \right)^{b_{1}b_{2} + b_{1} + b_{2}} e^{- \pi i \left( b_{1}a_{2} - b_{2}a_{1} \right)}$;
    \item[iii)]   $g_{\mathbf{a}} \left( \gamma \tau \right) = \zeta \left( \gamma \right) g_{\mathbf{a} \gamma } ( \tau )$ where $ \zeta \left( \gamma \right)$ is a $12$th root of unity, depending only on $\gamma$.
    \end{itemize}
\end{prop}
\begin{proof}
 Parts (i) and (ii) follow from the definition of $g_{\mathbf{a}} ( \tau )$ and the results proved in \cite[Pages 27-29]{KL}. Part (iii) follows from the definition of $g_{\mathbf{a}} ( \tau )$ and the transformation formula for the Dedekind eta function proved in  \cite{newman1957construction}.
 \end{proof}

It follows from the above transformations and some elementary algebra that $g_{\mathbf{a}}(\tau)^{12p}$ is invariant under the action of $\Gamma(p)$. Moreover, as argued in \cite[Pages 30-31]{KL}, $g_{\mathbf{a}}(\tau)$  has no zeros or poles in the upper half plane, and hence $g_{\mathbf{a}}(\tau)^{12p}$ is a modular unit on $X(p)$. From the product expansion above we deduce the following. 
\begin{prop} \label{ordga}
For any  $\mathbf{a} = \left( a _{1}, a_{2} \right) \in \Q^{2} \setminus \Z^{2}$ 
\begin{center}
    $\text{ord}_{q}\left( g_{\mathbf{a}} ( \tau ) \right) = \frac{1}{2} B_{2} \left( \{ a_{1} \} \right)$
\end{center}
where $B_{2}$ is the second Bernoulli polynomial as defined above and  $\{ x \} \in \left[ 0 , 1 \right)$ denotes the fractional part of $x \in \mathbb{R}$.
\end{prop}
 
\subsection{Modular Units on \texorpdfstring{$X( p )$}{d}}
Kubert and Lang \cite[Chapter 4, Theorem 1.3]{KL} proved that if $N$ is coprime to $6$, the modular units on $X(N)$ are products of Siegel functions. 
\begin{thm}{(Kubert, Lang)} \label{KLthm}
  Let $N \in \mathbb{N}$ be such that $\text{gcd}\left( N , 6 \right) =1$ and let  
  \begin{equation}
f ( \tau ) = \displaystyle \prod_{\mathbf{a}} g_{\mathbf{a}}( \tau )^{e_{\mathbf{a}}}
  \end{equation}
 where the product is taken over all $\mathbf{a} \in \frac{1}{N} \Z^{2} / \Z^{2}$, and the exponents  are integers satisfying the following congruence relations:
  \begin{itemize}
      \item $\displaystyle \sum_{\mathbf{a}} e_{\mathbf{a}} \left( a_{1}N \right)^{2} \equiv 0 \mod{N}; $
      \item $ \displaystyle \sum_{\mathbf{a}} e_{\mathbf{a}} \left( a_{2}N \right)^{2} \equiv 0 \mod{N}; $
\item $ \displaystyle \sum_{\mathbf{a}} e_{\mathbf{a}} \left( a_{1} N \right) \left( a_{2} N \right) \equiv 0 \mod{N};  $ 
      \item $\displaystyle \sum_{\mathbf{a}} e_{\mathbf{a}} \equiv 0 \ \text{mod} \ 12$.
  \end{itemize}
  Then $f$ is a modular unit on $X \left( N \right)$. 
Furthermore, all modular units on $X\left( N \right)$ are, up to scaling, of the above form.
\end{thm}

In particular, this results fully describes the modular units on $X(p)$, for prime level $p \ge 5$. As $\Gamma(p) \subset \Gamma_{1}(p)$, modular units on $X_{1}(p)$ are necessarily modular units on $X(p)$, and hence of the above form. In Section \ref{constructions} we determine when such products of Siegel functions are additionally invariant under $\Gamma_{1}(p)/\Gamma(p)$.

\section{Constructing Modular Units on \texorpdfstring{$X_{1}(p)$}{a}} \label{constructions}
The main result of this section gives a necessary and sufficient criteria for certain products of Siegel functions to be modular units on the curve $X_{1}( p)$. In the proceeding section, we  prove that all modular units on $X_{1}\left( p\right)$ are of this form. 

For the rest of this paper we fix an integer $\alpha$ whose reduction modulo $p$ has multiplicative order $p-1$ in $\left( \Z / p \Z \right)^{*}$. All of our constructions use this fixed $\alpha$. We set  $n = \frac{p-1}{2} \in \Z$ and for $0 \le i \le n-1$ define:
\begin{itemize}
    \item $E_{i} ( \tau ) = g_{\left( \frac{\alpha^{i}}{p}, 0 \right)} ( \tau ) \displaystyle \prod_{j=0}^{p-2} g_{\left( \frac{\alpha^{i}}{p}, \frac{\alpha^{j}}{p} \right)}( \tau ); $
    \item $F_{i} ( \tau ) = g_{\left( 0, \frac{\alpha^{i}}{p} \right)} ( \tau )$.
\end{itemize}
The main result of this section is the following.
\begin{prop} 
\label{X1pcongs}
Let $e_{i}, f_{j} \in \Z$ and set 
\begin{equation} \label{fx1p}
    f ( \tau) = \displaystyle \prod_{i=0}^{n-1} E_{i}(\tau)^{e_{i}} F_{i} ( \tau)^{f_{i}}.
 \end{equation}   
 Then $f$ is a modular unit for $\Gamma_{1}( p )$ if and only if the exponents satisfy the following congruence conditions:
\begin{itemize}
    \item $ \displaystyle \sum_{i=0}^{n-1} e_{i} \alpha^{2i}  \equiv 0 \mod{p};$
    \item $\ \displaystyle \sum_{i=0}^{n-1} f_{i} \alpha^{2i}  \equiv 0 \mod{p};$
    \item $ \displaystyle \sum_{i=0}^{n-1} pe_{i} + f_{i} \equiv \ 0 \mod{12}.$
\end{itemize}
\end{prop}

We first note that in the notation of Theorem \ref{KLthm}, for any $f$ as \eqref{fx1p}:
\begin{itemize}
    \item $\displaystyle \sum_{\mathbf{a}} e_{\mathbf{a}} \left( a_{1}N \right)^{2} = \displaystyle \sum_{i=0}^{n-1} p e_{i} \alpha^{2i} \ \equiv \ 0 \ \text{mod} \ p,$
     \item $\displaystyle \sum_{\mathbf{a}} e_{\mathbf{a}} \left( a_{2}N \right)^{2} =  \displaystyle \sum_{i=0}^{n-1} e_{i} \sum_{j=0}^{p-2} \alpha^{2j} + \displaystyle \sum_{i=0}^{n-1}  f_{i} \alpha^{2i} \equiv \ \displaystyle \sum_{i=0}^{n-1}  f_{i} \alpha^{2i} \ \equiv \ 0 \ \text{mod} \ p,$    
      \item $\displaystyle \sum_{\mathbf{a}} e_{\mathbf{a}} \left( a_{1} N \right) \left( a_{2}N \right) =  \displaystyle \sum_{i=0}^{n-1} e_{i} \sum_{j=0}^{p-2} \alpha^{i + j}  \ \equiv \ 0 \ \text{mod} \ p,$        
  \item $\displaystyle \sum_{\mathbf{a}} e_{\mathbf{a}} =  \displaystyle \sum_{i=0}^{n-1} pe_{i} + f_{i} \ \equiv \ 0 \ \text{mod} \  12,$
  \end{itemize}
and hence $f$ is a modular unit for $\Gamma ( p)$. Therefore, to prove that such an $f$ is a modular unit for $\Gamma_{1} ( p)$ it suffices to show
\begin{center}
$f ( \gamma \tau) = f ( \tau)$
\end{center}
where $\gamma$ runs through any set of coset representatives of $\Gamma_{1}( p) / \Gamma ( p)$. Our argument uses the following preliminary results; the first of which fixes a good choice of representatives with our notation, and the second proves transformation formulae for $E_{i}$ and $F_{j}$ under these representatives.
\begin{lem} \label{cosetG1pGp}
Let $\gamma_{k} = \begin{psmallmatrix}
    1 & \alpha^{k} \\ 0 & 1 
\end{psmallmatrix}$  for $k = 0, \ldots p-2$. Then $ \{ \gamma_{0}, \ldots, \gamma_{p-2} ,I_{2 \times 2} \}$ is a complete set of coset representaives for $\Gamma_{1} ( p) / \Gamma ( p)$.
\end{lem}

\begin{proof}
For $\gamma = \begin{psmallmatrix}
    a & b \\ c & d
\end{psmallmatrix}\in \Gamma_{1}( p )$  with $b$ coprime to $p$, $b  \equiv \alpha^{k} \mod{p}$ for some $0 \le k \le p-2$, and one shows that $\gamma \gamma_{k}^{-1} \in \Gamma ( p )$.
\end{proof}

\begin{prop} \label{transFE}
Let $\gamma = \gamma_{k}$ for some $k \in \{ 0, \ldots, p-2 \}$ be as in Lemma \ref{cosetG1pGp}. For any $i \in \{ 0, \ldots n-1 \}$
\begin{itemize}
    \item $ E_{i}( \gamma \tau ) = \zeta \left( \gamma \right)^{p} (-1)^{\alpha^{i+k}}\exp \left(  \frac{\pi i \alpha^{2i+k}}{p}\right) \ E_{i}( \tau ), $
    \item $ F_{i}( \gamma \tau ) = \zeta \left( \gamma \right) F_{i} ( \tau ),$
\end{itemize}
where $\zeta \left( \gamma \right)$ is a $12$th root of unity, which depends on $\gamma$ only (c.f Proposition \ref{trans2}).
\end{prop}
\begin{proof}
The result is immediate for $F_{i}$ by Proposition \ref{trans2}. From the definition of $E_{i}$ and Proposition \ref{trans2} we obtain 
\begin{align*}
  E_{i} ( \gamma \tau ) & = \zeta \left( \gamma \right)^{p} g_{( \frac{\alpha^{i}}{p}, 0 ) \gamma } ( \tau ) \displaystyle \prod_{j=0}^{p-2}  g_{( \frac{\alpha^{i}}{p}, \frac{\alpha^{j}}{p} ) \gamma } ( \tau ) \\
    & = \zeta \left( \gamma \right)^{p} g_{( \frac{\alpha^{i}}{p}, \frac{\alpha^{i+k}}{p} )} ( \tau ) \displaystyle \prod_{j=0}^{p-2}  g_{( \frac{\alpha^{i}}{p}, \frac{ \alpha^{i+k} + \alpha^{j}}{p})} ( \tau )
    \end{align*}
    where $\zeta \left( \gamma \right)$ is as before. We set  $\alpha^{\infty} := 0$ and note that for all $j \in \{ 0, \ldots, p-2 \} \cup \{ \infty \}$ we have 
\begin{center}
    $\alpha^{i+k} + \alpha^{j} =  \alpha^{m_{j}} + n_{j}p $
\end{center}
for some unique $m_{j} \in  \{ 0, \ldots, p-2 \} \cup \{ \infty \}$ and $n_{j} \in \Z $. Combining this with Proposition \ref{trans2} we obtain 
\begin{align*}
g_{( \frac{\alpha^{i}}{p}, \frac{\alpha^{i+k} + \alpha^{j}}{p}))} ( \tau ) & =  g_{( \frac{\alpha^{i}}{p}, \frac{\alpha^{m_{j}}}{p} ) + ( 0, n_{j})} ( \tau ) \\
& =  (-1)^{n_{j}} \exp \left( \pi i n_{j} \frac{\alpha^{i}}{p} \right)  g_{( \frac{\alpha^{i}}{p}, \frac{\alpha^{m_{j}}}{p} )} ( \tau ).
\end{align*}
Hence substituting this into the above expression, we obtain 
\begin{center}
   $E_{i} \left( \gamma \tau \right) = \zeta \left( \gamma \right)^{p} \left( -1 \right)^{\sum_{j} n_{j}} \exp \left( \pi i \frac{\alpha^{i}}{p} \sum_{j}n_{j} \right) \displaystyle \prod_{j} g_{( \frac{\alpha^{i}}{p}, \frac{\alpha^{m_{j}}}{p})} ( \tau )$. 
\end{center}
Note that the function $j \mapsto m_{j}$ defines a bijection on 
$ J =   \{ 0, \ldots, p-2 \} \cup \{ \infty \}$ and hence
\begin{center}
   $E_{i} \left( \gamma \tau \right) = \zeta \left( \gamma \right)^{p} \left( -1 \right)^{\sum_{j} n_{j}} \exp \left( \pi i \frac{\alpha^{i}}{p} \sum_{j}n_{j} \right) E_{i} ( \tau )$. 
\end{center}
Thus we conclude 
\begin{align*}
     \displaystyle \sum_{j \in J  } n_{j}   = \displaystyle \sum_{j \in J } \frac{1}{p} \left( \alpha^{i+k}  + \alpha^{j} - \alpha^{m_{j}} \right) 
     = \displaystyle \sum_{j \in J } \frac{\alpha^{i+k}} {p} = \alpha^{i+k}
      \end{align*}
 and the claim follows.     
\end{proof}

\begin{proof}[Proof of Proposition \ref{X1pcongs}]
As previously observed, to prove the forward direction, it remains to show that $ f ( \gamma \tau ) = f ( \tau )$ for any $\gamma = \gamma_{k}$ as in Lemma \ref{cosetG1pGp}. It follows from Proposition \ref{transFE} that 
    \begin{center}
       $f ( \gamma \tau ) = \zeta \left( \gamma \right)^{A} (-1)^{B} e^{\pi i C} f( \tau ) $ 
    \end{center}
    where $A = \displaystyle \sum_{i=0}^{n-1} pe_{i} + f_{i}$, $B = \alpha^{k} \displaystyle \sum_{i=0}^{n-1} \alpha^{i} e_{i}$ and $C = \frac{\alpha^{k}}{p} \displaystyle \sum_{i=0}^{n-1} \alpha^{2i} e_{i}$.
    
    As $A \equiv \ 0 \ \text{mod} \ 12 $ and $\zeta \left( \gamma \right)$ is a $12$th root of unity, $\zeta \left( \gamma \right)^{A} =1$. By hypothesis $ \displaystyle \sum_{i=0}^{n-1} \alpha^{2i} e_{i} \ \equiv \ 0 \  \text{mod} \ p $ so $C \in \Z$. Lastly, since $p$ is odd, $B \equiv C \ \text{mod} \ 2 $, and hence $(-1)^{B}  = e^{\pi i C}   \in \{ \pm 1 \}$. 

On the contrary, suppose that such an $f$ is a modular unit on $X_{1}\left( p\right)$. Then $f ( \gamma \tau ) = f ( \tau )$ for $\gamma = \begin{psmallmatrix}
     1 & p \\ 0 & 1 
\end{psmallmatrix} \in \Gamma_{1} ( p)$. From Lemma \ref{trans2}, we note that 
$f \left( \gamma \tau \right) = \zeta \left( \gamma \right)^{\sum pe_{i} + f_{i}} f ( \tau )$ and since $\zeta\left( \gamma \right)$ is a $12$th root of unity, we get 
\begin{center}
    $\displaystyle \sum_{i=0}^{n-1} p e_{i} + f_{i} \equiv 0 \mod{12}$.
\end{center}
Take $\gamma = \begin{psmallmatrix}
    1 & 0 \\ p \alpha & 1 
\end{psmallmatrix} \in \Gamma_{1}(p)$, then:
\begin{center}
$f( \gamma \tau ) = 
\zeta ( \tau )^{\sum pe_{i} + f_{i}} e^{-i \pi \alpha \sum f_{i} \frac{\alpha^{2i}}{p}}f (\tau)$
\end{center}
and this implies the second stated congruence. Finally, taking $\gamma = \begin{psmallmatrix}
     1 & 1 \\ 0 & 1 
\end{psmallmatrix}$ gives the first congruence.
\end{proof}

We spend the remained of this section discussing divisors of modular units of the form \eqref{fx1p}. Recall from Proposition \ref{cuspsX1p} that $X_{1}(p)$ has $p-1$ cusps, and we define explicit representatives of these using our fixed $\alpha$ as follows:
\begin{itemize}
    \item for $0 \le i \le n-1$ let $P_{i}$ be the cusps represented by $\begin{psmallmatrix}
        \alpha^{i} \\ p 
    \end{psmallmatrix}$, or equivalently by $\frac{\alpha^{i}}{p} \in \Q \cup \{ i \infty \}$;
\item for $0 \le i \le n-1$ let $Q_{i}$ be the cusps represented by $\begin{psmallmatrix} p \\
        \alpha^{i} 
    \end{psmallmatrix}$, or equivalently by $\frac{p}{\alpha^{i}} \in \Q \cup \{ i \infty \}$.
\end{itemize}
Note that $P_{0}, \ldots ,P_{n-1} \in X_{1}( p)(\Q )$ and $Q_{0}, \ldots, Q_{n-1}$ are defined over $\Q ( \zeta_{p})^{+}$, the maximal real subfield of the cyclotomic extension $\Q ( \zeta_{p})$, and form one orbit under the Galois action. Moreover $P_{0}, \ldots, P_{n-1}$ lie over the cusp represented by $i \infty$ on $X_{0}(p)$ and $Q_{0}, \ldots, Q_{n-1}$ lie over the cusp represented by $0$ on $X_{0}(p)$.

In what follows we write $(a_{1}, a_{2})_{1} := a_{1}$ and $\{ x \} = x - \floor{x}$ for the fractional part of $x \in \R$.
\begin{lem}
Suppose that $g (\tau ) = \displaystyle \prod_{\mathbf{a} \in I} g_{\mathbf{a}}(\tau )^{b_{\mathbf{a}}}$ is a modular unit on $X_{1}(p)$. Then 
\begin{center}
    $ \divv (g) =  \displaystyle \sum_{c} \frac{w_{c}}{2} \sum_{\mathbf{a} \in I} B_{2}(\{ (\mathbf{a}\gamma_{c})_{1}\})$ 
\end{center}
where the first sum is over the cusps of $X_{1}(p)$, $w_{c}$ is the width of the cusp $c$ and $\gamma_{c} \in \SL_{2}(\Z)$ is such that $\gamma_{c} \cdot i \infty = c$.  
\end{lem}
\begin{proof}
 This follows directly from Propositions \ref{trans2} and \ref{ordga}.   
\end{proof}
 
\begin{prop} \label{ordsvansX1}
Let $f( \tau ) = \displaystyle \prod_{i=0}^{s-1} E_{i} ( \tau )^{e_{i}} F_{i} ( \tau )^{f_{i}}$  be a modular unit on $X_{1} \left( p\right)$. Then for all $0 \le i \le n-1$:  
\begin{itemize}
    \item $ \text{ord}_{P_{i}} f   =  \displaystyle \sum_{j=0}^{n-1} e_{j} \frac{p}{2} B_{2} \left( \left \{ \frac{\alpha^{i+j}}{p} \right \} \right) + \sum_{j=0}^{n-1} \frac{f_{j}}{12};$ 
    \item  $ \text{ord}_{Q_{i}} f   =  \displaystyle \sum_{j=0}^{n-1} f_{j} \frac{p}{2} B_{2} \left( \left \{ \frac{\alpha^{i+j}}{p} \right \} \right) + \sum_{j=0}^{n-1} \frac{e_{j}}{12}$.
    \end{itemize}
\end{prop}
\begin{proof}
The cusps $P_{i}$ all have width $1$, whilst all $Q_{j}$ have width $p$. We further note that for all $0 \le i \le n-1$, there exist integers $a_{i}, b_{i} $ with $\gamma_{P_{i}} = \begin{psmallmatrix}
    \alpha^{i} & a_{i} \\
    p & b_{i}
\end{psmallmatrix}$  and $\gamma_{Q_{i}} = \begin{psmallmatrix}
    p & -b_{i} \\ \alpha^{i} & -a_{i}
\end{psmallmatrix} $. The results follows from the previous lemma. 
\end{proof}

\section{Proof of Theorem \ref{basisx1p2}} \label{basissec}
In this section we use the criteria proved Proposition \ref{X1pcongs} to construct a basis for the group of modular units on $X_{1}(p)$. We write $\cF( p)$ for the group of modular units for $\Gamma_{1}( p)$ and $\Div_{c} (p)$ for the group of degree $0$ divisors on $X_{1}( p)$ supported on cusps. As $X_{1}( p)$ has $p-1$ cusps, $\Div_{c}(p)$ is a free abelian group of rank $p-2$. The homomorphism 
\begin{center}
    $\divv : \cF( p ) \longrightarrow \Div_{c}(p)$
\end{center}
is an injection when restricted to $\cF( p)$ modulo constants, and hence $\cF(p) / \mathbb{C}$ is a free abelian group of rank at most $p-2$. This upper bound is in fact attained, which we prove in Theorem \ref{basisx1p2}.

The \textit{cuspidal group of} $X_{1}( p)$ is the quotient group 
\begin{center}
    $\calC_{1}( p) = \Div_{c}( p) / \langle \divv (f)  : f \in \cF(p) \rangle$.
\end{center}
This is a finite subgroup of $J_{1} ( p)$, the Jacobian of $X_{1} ( p)$, due to the theorems of Manin \cite{Manin} and Drinfeld \cite{Drinfeld}. The order of this group, denoted by $h_{1}( p)$, if often referred to as the \textit{cuspidal class number} of $X_{1} \left( p\right)$. This is known in our case     due to works of Takagi \cite{takagi1992cuspidal} and Yu \cite{yu1980cuspidal}.

\begin{thm} \label{cuspidalclassnumberX1p}
The cuspidal class number $h_{1} ( p)$ of $X_{1} ( p )$ is  
\begin{center}
    $ h_{1} ( p ) =  \left( p \displaystyle \prod_{\chi} \frac{1}{4} B_{2, \chi} \right)^{2} $
\end{center}
where the product runs over all even, non-principal, Dirichlet characters modulo $p$, $B_{2,\chi} = p \displaystyle \sum_{a =1}^{p} \chi \left( a \right) B_{2} \left( \frac{a}{p} \right)$, and $B_{2} \left( x \right) = x^{2} - x + \frac{1}{6}$ is the second Bernoulli polynomial.
\end{thm}

\begin{proof}
In \cite[Theorem 4.1]{takagi1992cuspidal} Takagi proves that
\begin{center}
    $h_{1}\left( p\right) = \left( h_{1}^{\infty} \left( p\right) \right)^{2}$
\end{center}
where $h_{1}^{\infty} ( p )$ is the size of $\calC_{1}^{\infty} ( p )$, the subgroup of $J_{1}\left( p\right)$ generated by the cusps $P_{0}, \ldots, P_{n-1}$. In \cite[Theorem 5]{yu1980cuspidal}, Yu computed $h_{1}^{\infty} ( p )$, which is equal to $p \displaystyle \prod_{\chi} \frac{1}{4} B_{2, \chi}$, as proved by Yang in \cite[Theorem $A^{'}$]{yang2007modular}.
\end{proof}

Fix $\beta \in \{ \pm 1, \pm 5 \}$ with $p \equiv  \beta \mod{12} $. For $i \in \{0, \ldots, n-2 \}$, define
\begin{align}
    & G_{i} ( \tau  ) = E_{i} ( \tau  ) E_{n-1} ( \tau  )^{-\alpha^{2i+2}}F_{n-1} ( \tau  )^{p (\alpha^{2i+2} -1)}; \\
    & H_{i} ( \tau  ) = F_{i}( \tau  ) F_{n-1}( \tau  )^{ -\alpha^{2i+2} + p \beta (\alpha^{2i+2} -1)};
\end{align}
and further define 
\begin{align}
    & G_{n-1} ( \tau  ) = E_{n-1} ( \tau  )^{p} F_{n-1} ( \tau  )^{-\beta p}; \\
    & H_{n-1} ( \tau  ) = F_{n-1}( \tau  )^{12p}.
\end{align}

All of the functions above are modular units for $\Gamma_{1}( p )$ by Proposition \ref{X1pcongs}. Moreover their exponents correspond to a basis for the set of integer solutions to the system of congruences stated in Proposition \ref{X1pcongs}. In this section, we prove Theorem \ref{basisx1p2}, namely, we show that a certain subset of the above functions represents a basis for the group of modular units on $X_{1}(p)$, up to constants.  Our proof uses the following elementary results. For the rest of this 
section we write $a_{i} =  \frac{p}{2}B_{2} \left( \left \{ \frac{\alpha^{i}}{p} \right \} \right)$ with $0 \le i \le n-1$.
\begin{lem} \label{columnsum}
 $ \displaystyle \sum_{i=0}^{n-1} a_{i} =  -\frac{n}{12}  $.
\end{lem}
\begin{proof}
    Observe that $B_{2} \left( \{ -x \} \right)  = B_{2} \left( \{ x \} \right)$ and hence 
\begin{align*}
\displaystyle \sum_{i=0}^{n-1} B_{2} \left( \left \{ \frac{\alpha^{i}}{p} \right \} \right)  & = \displaystyle \sum_{x=1}^{n} B_{2} \left( \frac{x}{p} \right) = \frac{1}{6p^{2}}n(n+1)(2n+1) - \frac{1}{2p}n(n+1) + \frac{1}{6}n = \frac{-n}{6p}.
\end{align*}
\end{proof}

\begin{lem} \label{dets}
Let $a_{i}$ be as above, and let $N$ be the $n \times n$ circulant matrix 
\begin{center}
    $ N = \begin{bmatrix}
        a_{0} - \frac{1}{12} & a_{1} - \frac{1}{12} & \dots & a_{n-1} - \frac{1}{12} \\
         a_{1} - \frac{1}{12} & a_{2} - \frac{1}{12} & \dots & a_{0} - \frac{1}{12} \\
         \vdots & \vdots & \ddots & \vdots \\
         a_{n-1} - \frac{1}{12} & a_{0} - \frac{1}{12} & \dots & a_{n-2} - \frac{1}{12} \\            \end{bmatrix}.$
\end{center}
Then $\text{det}\left( N \right) = \frac{-n}{24} \displaystyle \prod_{\chi \ne \chi_{0}}  B_{2, \chi}$, where the product is taken over all even, non-principal Dirichlet characters modulo $p$ and $B_{2,\chi} = p \displaystyle \sum_{a=1}^{p} \chi \left( a \right) B_{2} \left( \frac{a}{p} \right)$ as before.
\end{lem}

\begin{proof}
    This proof is standard, and a similar result is proved in \cite{yang2007modular} in a comparable manner. Let $V$ be the complex vector space of $\mathbb{C}-$valued even functions  $f : \left( \Z / p \Z \right)^{*} \longrightarrow \mathbb{C}$. For $i \in \{ 1, \ldots, n \}$, define the linear operators on $V$
\begin{center}
    $T_{i} f (x) = f ( \alpha^{i-1} x \ \text{mod} \ p )$
\end{center}
and consider the operator 
\begin{center}
    $T = \displaystyle \sum_{i=1}^{n} \left( a_{i-1} - \frac{1}{12} \right) T_{i}$.
\end{center}

There are two standard bases of $V$. Firstly, $\delta_{1}, \ldots, \delta_{n}$, where 
 \begin{equation*}
\delta_{i} (x) = 
\begin{cases}
    1 & \text{if } x \equiv \ \pm  \alpha^{i-1} \ \text{mod} \ p  \\ 
   0  & \text{otherwise }
\end{cases}
\end{equation*}
We note that $N$ is the matrix of $T$ with respect to this basis, up to some ordering of the basis $\delta_{1} , \ldots, \delta_{n}$. 
 
The second basis of $V$ is the set of even Dirichlet characters modulo $p$. If $\chi$ is a non-principal even Dirichlet character, 
\begin{align*}
    T \chi ( x) & = \displaystyle \sum_{i=1}^{n} ( a_{i-1} - \frac{1}{12}) \chi \left( \alpha^{i-1} \right) \chi \left( x \right) \\
           & =  \displaystyle \sum_{i=1}^{n} a_{i-1} \chi \left( \alpha^{i-1} \right) \chi \left( x \right) \\
           & =  \displaystyle \sum_{i=1}^{n}  \frac{1}{12} \chi \left( \alpha^{i-1} \right) \chi \left( x \right)     \\
 & =   \frac{1}{2} p \displaystyle \sum_{a=1}^{p} \chi \left( a \right) B_{2} \left( \frac{a}{p} \right) \chi \left( x \right)
\end{align*}
and hence $ T \chi \left( x \right) = \frac{1}{4} B_{2, \chi } \chi \left( x \right)$.
If $\chi = \chi_{0}$ is the principal Dirichlet character modulo $p$, then 
\begin{align*}
    T \chi_{0} \left( x \right) &= \displaystyle \sum_{i=1}^{n} \left( a_{i-1} - \frac{1}{12} \right) T_{i} \chi_{0} \left( x \right) \\
     &= \displaystyle \sum_{i=1}^{n} \left( a_{i-1} - \frac{1}{12} \right) \chi_{0} \left( x \right) \\
      &=  \left(\frac{p}{12 }\displaystyle \sum_{a=1}^{n} B_{2} \left( \frac{a}{p} \right) - \frac{n}{12} \right) \chi_{0} \left( x \right) \\
      & = \frac{-2n}{12} \chi_{0} \left( x \right)
        \end{align*}
where the last equality follows from Lemma \ref{columnsum}. The matrix of $T$ with respect to this second basis is diagonal, and has determinant $\frac{-2s}{12} \displaystyle \prod_{\chi \ne \chi_{0}} \frac{1}{4} B_{2, \chi}$.  
\end{proof}
\begin{proof}[Proof of Theorem \ref{basisx1p2}]
    Let $\Lambda \cong \mathbb{Z}^{p-2}$ be the lattice generated by differences of cusps. We take $P_{0} - Q_{n-1},\ldots, P_{n-1} - Q_{n-1}, Q_{0} - Q_{n-1}, \ldots, Q_{n-2} - Q_{n-1}$ as a basis of $\Lambda$. Let  $\Lambda'$ be the sub-lattice generated by the divisors of $G_{0}, \ldots, G_{n-1}, H_{1}, \ldots, H_{n-1}$ viewed as elements of $\Lambda$. We denote by $M$ the matrix whose columns are the coefficients of these divisors (with respect to the fixed basis of $\Lambda$). Observe that 
    \begin{center}
        $\left( \Lambda : \Lambda' \right) = \vert \det\left( M \right) \vert$
    \end{center}
and thus it suffices to show $ \vert \det M \vert = h_{1}\left( p\right)$. Note that $M = AB $ where:
\begin{center}
    $ A = \begin{bmatrix}
        1 & 0 & \dots & 0 & - \alpha^{2} & 0 & \dots & 0 &  p(\alpha^{2} -1 ) \\
          0 & 1 & \dots & 0 &  - \alpha^{4} & 0 & \dots & 0 & p(\alpha^{4} -1 ) \\
          \vdots & \vdots & \ddots & \vdots & \vdots & \vdots & \ddots & \vdots & \vdots \\
          0 & 0 & \dots & 1 &  - \alpha^{2n-2} & 0 & \dots & 0 & p(\alpha^{2n -2} -1 ) \\
          0 & 0 & \dots & 0 &  p & 0 & \dots & 0 & -\beta p \\
          0 & 0 & \dots & 0 &  0 & 1 & \dots & 0 &  - \alpha^{4} + p \beta (\alpha^{4} -1 ) \\
           \vdots & \vdots & \ddots & \vdots & \vdots & \vdots & \ddots & \vdots & \vdots \\
           0 & 0 & \dots & 0 &  0 & 0 & \dots & 1 &  - \alpha^{2n-2} + p\beta (\alpha^{2n-2} -1 ) \\   
            0 & 0 & \dots & 0 &  0 & 0 & \dots & 0 &  12p \\              \end{bmatrix}$
\end{center}
\begin{center}
    $ B = \begin{bmatrix}
       a_{0} & a_{1} & \dots & a_{n-1} & \frac{1}{12} & \frac{1}{12} & \dots & \frac{1}{12} \\
      a_{1} & a_{2} & \dots & a_{0} & \frac{1}{12} & \frac{1}{12} & \dots & \frac{1}{12} \\
        \vdots  & \vdots  & \ddots & \vdots & \vdots & \vdots  & \ddots & \vdots \\
       a_{n-1} & a_{0} & \dots & a_{n-2} & \frac{1}{12} & \frac{1}{12} & \dots & \frac{1}{12} \\      
        \frac{1}{12} & \frac{1}{12} & \dots & \frac{1}{12} & a_{1} & a_{2} & \dots & a_{n-1} \\ 
         \frac{1}{12} & \frac{1}{12} & \dots & \frac{1}{12} & a_{2} & a_{1} & \dots & a_{0} \\ 
         \vdots  & \vdots  & \ddots & \vdots & \vdots & \vdots  & \ddots & \vdots \\
\frac{1}{12} & \frac{1}{12} & \dots & \frac{1}{12} & a_{n-1} & a_{0} & \dots & a_{n-3}  \\ 
\end{bmatrix}$.
\end{center}
The determinant of $A$ is clearly $12p^{2}$. To compute the determinant of $B$ we consider the $(p-1) \times (p-1)$ block matrix:
\begin{center}
    $J = \begin{bmatrix}
        a_{0}  - \frac{1}{12} & \dots & a_{n-1} - \frac{1}{12} & 0 & \dots & 0 \\
        \vdots &  \ddots & \vdots &  \vdots &  \ddots & \vdots \\
        a_{n-1} - \frac{1}{12} & \dots & a_{n-2} -   \frac{1}{12} & 0  & \dots & 0 \\
        0  &  \dots & 0  & a_{0}  - \frac{1}{12} & \dots & a_{n-1} - \frac{1}{12} \\
         \vdots & \ddots & \vdots & \vdots &  \ddots & \vdots \\
         0 &  \dots & 0  & a_{n-1} - \frac{1}{12} &  \dots & a_{n-2} - \frac{1}{12} \\
          \end{bmatrix}$\end{center}

and we perform the following sequence of  row and column operations on $J$:
\begin{itemize}
    \item[1.] add all rows to the $(n+1)$th row;
    \item[2.] add all columns to the last one;
    \item[3.] factor out $-2n$ from the $n+1$th row;
    \item[4.] add the $n+1$th row to all other rows.
    \end{itemize} 
Applying Lemma \ref{columnsum}, we the resulting matrix is:
\begin{center}
    $ -2n \cdot \begin{bmatrix}
        a_{0}  & \dots & a_{n-1} & \frac{1}{12} & \dots & 0 \\
        \vdots &  \ddots & \vdots &  \vdots &  \ddots & \vdots \\
        a_{n-1}  & \dots & a_{n-2}  & \frac{1}{12}  & \dots & 0  \\
        \frac{1}{12}  &  \dots & \frac{1}{12}  &   \frac{1}{12} & \dots &  \frac{2n}{12} \\
         \vdots & \ddots & \vdots & \vdots &  \ddots & \vdots \\
         \frac{1}{12} &  \dots & \frac{1}{12} &  a_{n-1} &  \dots &  0 \\
          \end{bmatrix}$.
\end{center}
Expanding by the last column, the determinant of $J$ is:
\begin{center}
    $\text{det} \left( J \right) = -2n \left( \frac{2n}{12} \right) \text{det} \left( B \right)$.
\end{center}
Since $J$ is a block matrix, we find 
\begin{center}
$\text{det}\left( J \right) = \text{det} \left( N \right)^{2}$
\end{center}
where $N$ is the matrix defined in Lemma \ref{dets}. Thus 
\begin{center}
    $ \text{det} \left( B \right) = \frac{-1}{12} \displaystyle \prod_{\chi \ne \chi_{0}} \frac{1}{4} B_{2,\chi}$,
\end{center}
and hence 
\begin{center}
    $ \vert \text{det} \left( M \right) \vert = \vert \text{det} \left( A \right) \vert \vert \text{det} \left( B \right) \vert = p^{2} \vert \displaystyle \prod_{\chi \ne \chi_{0} } B_{2, \chi} \vert^{2} $
\end{center}
and so our result follows from Theorem \ref{cuspidalclassnumberX1p}. 
\end{proof}

We remarked that although the function $H_{0}( \tau )$ was defined at the beginning of this section, it does not appear in our basis. We give the explicit presentation of this function in terms of our basis and conclude with a corollary which determines different bases of $\cF( p )$ containing $H_{0}$.

 Consider the function 
 \begin{align*}
     F ( \tau) & = \displaystyle \prod_{i=0}^{n-2}G_{i} ( \tau) H_{i} ( \tau)  = f ( \tau) E_{n-1}( \tau)^{C-1}F_{n-1}( \tau )^{D-1}
 \end{align*}
 where $f ( \tau) = \displaystyle \prod_{i=0}^{n-1}E_{i}( \tau ) F_{i} ( \tau )$ and the integer constants $C,D$ are defined as
 \begin{itemize}
    \item $C  =  \displaystyle \sum_{i=0}^{n-2} - \alpha^{2i +2} = \frac{\alpha^{2n} - \alpha^{2}}{1 - \alpha^{2}}; $
   \item $ D = \displaystyle \sum_{i=0}^{n-2} p \left( \alpha^{2i+2} -1 \right) - \alpha^{2i+2} + p \beta \left( \alpha^{2i+2} -1 \right) - p \beta \left( n-1 \right)$.
\end{itemize}
 The function $f( \tau )$ is a modular unit for $\Gamma_{1}( p )$ by Proposition \ref{X1pcongs} and since the order of vanishing at any cusp is $0$, it is necessarily a constant function. Let $C_{1}$ be this constant value. We further observe that $C -1$ is divisible by $p$ and let $C_{2} = \frac{C-1}{p}$. Then 
 \begin{center}
     $F ( \tau ) = C_{1} G_{n-1} ( \tau)^{C_{2}} F_{n-1}( \tau )^{\beta\left( C -1 \right) + D -1}$.
 \end{center}
 By our choice of $\beta$, $\beta(C -1) + D -1 \ \equiv \ 0 \mod{12p} $, and thus 
 \begin{center}
     $F ( \tau) = C_{1} G_{n-1} ( \tau)^{C_{2}} H_{n-1} ( \tau)^{C_{3}}$
 \end{center}
 where $C_{3} = \frac{1}{12p} \left(\beta(C -1) + D -1\right) \in \Z$. We deduce the following. 
 \begin{cor} \label{altbasis}
  Any subset of $\{ G_{0} ( \tau ), \ldots, G_{n-1}( \tau), H_{0} ( \tau), \ldots, H_{n-1} ( \tau ) \}$ consisting of $p-2$ functions and containing $G_{n-1}$ and $H_{n-1}$ is a basis for the group of modular units for $\Gamma_{1} ( p )$ modulo constants. 
 \end{cor}
  
\section{Subgroups of Modular Units} \label{subgroups}
There are two natural subgroups of modular units which one is lead to consider when studying the rational cuspidal group $\mathcal{C}_{1}(p)(\Q)$. In this section we write 
$\Div_{c}^{\infty}(p)$ for the group of degree zero divisors supported on $P_{0}, \ldots, P_{n-1}$ and $\Div_{c}^{\Q}(p)$ for the group of degree $0$ divisors supported on cusps and fixed by the action of Galois. Furthermore, we denote by $\mathcal{F}^{\infty}(p)$ and $\mathcal{F}^{\Q}(p)$ the subgroups of $\mathcal{F}(p)$ consisting of functions with divisor in $\Div_{c}^{\infty}(p)$ and 
 $\Div_{c}^{\Q}(p)$ respectively.  The arising quotients:
 \begin{align*}
     & \mathcal{C}_{1}^{\infty} (p) =
     \Div_{c}^{\infty}(p) / \langle \divv(f) : f \in \mathcal{F}^{\infty}(p) \rangle \\ 
     & \mathcal{C}_{1}^{\Q}(p) := \Div_{c}^{\Q}(p) / \langle \divv(f) : f \in \mathcal{F}^{\Q}(p) \rangle      
      \end{align*}
are natural subgroups of the \textit{rational cuspidal subgroup} $\mathcal{C}_{1}(p)(\Q)$, that is the subgroup consisting of elements of $\mathcal{C}_{1}(p)$ fixed by the action of Galois, and we have the clear inclusions:
\begin{center}
    $\mathcal{C}_{1}^{\infty}(p) \subseteq \mathcal{C}_{1}^{\Q} (p) \subseteq \mathcal{C}_{1}(p)(\Q).$
\end{center}

In this this section we construct explicit bases for $\mathcal{F}^{\infty}(p)$ and $\mathcal{F}^{\Q}(p)$ modulo constants and prove the the above inequalities are in fact equalities.

\subsection{Proof of Theorem \ref{basis2}}
In \cite[Theorem 1]{yang2007modular} Yang finds a basis $f_{1}, \ldots, f_{n-1}$ for $\cF^{\infty}(p)$, constructed using $n-1$ functions $\Tilde{E}_{i} ( \tau  )$, which are defined as products of scaled  Siegel functions. We  construct $n-1$ functions $I_{1}, \ldots I_{n-1}$ and prove that their orders of vanishing at each cusps coincide with those of the $f_{i}$. More precisely, let  $\gamma = \alpha^{p-2}$ and define 
\begin{align*}
 I_{i} ( \tau  ) & = \frac{E_{i-1} ( \tau  ) E_{i+1} ( \tau  )^{\gamma^{2}}}{E_{i} ( \tau  )^{ \gamma^{2} +1}}
\end{align*}
for $1 \le i \le n-2$ and  $I_{n-1} ( \tau  ) = \left(\frac{E_{n-2} ( \tau  )}{E_{n-1} ( \tau  )} \right)^{p}$.

\begin{proof}[Proof of Theorem \ref{basis2}]
The basis constructed in  \cite[Theoren 1]{yang2007modular} is as follows: for $1 \le i \le n-2$
\begin{align*}
 f_{i} ( \tau  ) & = \frac{\Tilde{E}_{i-1} ( \tau  ) \Tilde{E}_{i+1} ( \tau  )^{\gamma^{2}}}{\Tilde{E}_{i} ( \tau  )^{ \gamma^{2} +1}}
\end{align*} and 
\begin{align*}
f_{n-1} ( \tau  ) & = \left(\frac{\Tilde{E}_{n-2} ( \tau  )}{\Tilde{E}_{n-1} ( \tau  )} \right)^{p}.
\end{align*}
where $\Tilde{E}_{i} ( \tau  ) = E_{\alpha^{i}}^{(p)}(\tau) = -g_{(\frac{\alpha^{i}}{p}, 0)}(p \tau)$, as defined in \cite[Page 518]{yang2007modular}.

From \cite[Proposition 3]{yang2007modular} we deduce that for $1 \le i \le n-2$ and $0 \le k \le n-1$
\begin{align*}
 & \text{ord}_{P_{k}} \left( f_{i} \right) = \frac{p}{2} B_{2} \left( \left \{ \frac{\alpha^{i-1+k}}{p} \right \} \right) + \frac{p \gamma^{2}}{2} B_{2} \left( \left \{ \frac{\alpha^{i+1+k}}{p} \right \} \right) - \frac{p (\gamma^2 +1) }{2} B_{2} \left( \left \{ \frac{\alpha^{i+k}}{2} \right \} \right),   \\
 & \text{ord}_{Q_{k}} \left( f_{i} \right) = 0,
 \end{align*}
 and similarly, for all $0 \le k \le n-1$
 \begin{align*}
 & \text{ord}_{P_{k}} \left( f_{n-1} \right) = \frac{p^2}{2} \left( B_{2} \left( \left \{ \frac{\alpha^{n-2+k}}{p} \right \} \right) - B_{2} \left( \left \{ \frac{\alpha^{n-1+k}}{p} \right \} \right) \right),  \\
 & \text{ord}_{Q_{k}} \left( f_{n-1} \right) = 0.
 \end{align*}
Using Proposition \ref{ordsvansX1}, we verify that the orders of vanishing of $I_{1}, \ldots, I_{n-1}$ at each cusp coincide with the above.
 \end{proof}

\subsection{Proof of Theorem \ref{basis3}}
Notably, $\cF^{\infty}(p)$ is a subgroup of $\cF^{\Q}(p)$, and we extend the basis $I_{1}, \ldots, I_{n-1}$ constructed in the previous section to basis of $\mathcal{F}^{\Q}(p)$ as follows. 
In our choice of model the cusps $P_{i}$ are rational and the $Q_{j}$ form a single orbit under the Galois action. For  $f \in \cF( p )$, its divisor is fixed by the action of Galois if and only if it is of the form 
\begin{equation} \label{Qdivisors}
     \divv  f  = \displaystyle \sum_{i=0}^{n-1} d_{i}P_{i} + d \displaystyle \sum_{i=0}^{n-1} Q_{i}
\end{equation}
for integers $d_{0}, \ldots, d_{n-1}, d$. 

When $p \equiv 1 \mod{4}$, define:
\begin{equation} \label{Cs}
    I_{n} \left( \tau \right) = 
        \left( E_{0} \left( \tau \right) E_{\frac{n}{2}} \left( \tau \right) \right)^{6}.
\end{equation}
When $p \ \equiv \ 3 \mod{4}$, let $m \in \Z $ be such that 

\begin{equation} 
   \alpha^{2m} \equiv  \begin{cases}
       -2  \mod{p} & \text{if} \ \ p  \equiv 3 \mod{8} \\
         2 \mod{p} & \text{if} \ \  p  \equiv 7 \mod{8}.    \end{cases}
\end{equation}
We assume, without loss of generality, that $m$ is the unique integer with the above property and $0 \le m < n$. Define 
\begin{equation} 
   I_{n} \left( \tau \right) =  \begin{cases}
       E_{0} \left( \tau \right)^{8} E_{m} \left(\tau \right)^{4}  \ & \text{if} \  p  \equiv 3 \mod{8} \\
         E_{0} \left( \tau\right)^{24} E_{m} \left( \tau \right)^{-12} \ & \text{if} \  p  \equiv 7 \mod{8}.
         \end{cases}
\end{equation}
It follows from Proposition \ref{X1pcongs} that in all cases $I_{n}$ is a function on $X_{1} ( p )$ and using Proposition \ref{ordsvansX1} we deduce that its order of vanishing at each $Q_{i}$ is $1$. We prove that $I_{1}, \ldots, I_{n}$ is a basis for $\mathcal{F}^{\Q}(p)$ modulo constants. 
\begin{proof}[Proof of Theorem \ref{basis3}]
The  divisor of any $f \in \cF^{\Q} (p)$ is necessarily of the form 
 \begin{equation} 
     \divv ( f ) = \displaystyle \sum_{i=0}^{n-1} d_{i}P_{i} + d \displaystyle \sum_{i=0}^{n-1} Q_{i}
\end{equation}
for an integer $d$. As the order of vanishing of $I_{n}$ at each $Q_{i}$ is $1$, we have $f / I_{n}^{d} \in \cF^{\infty} ( p)$; and hence, up to scaling, $f$ is in the span of $I_{1}, \ldots, I_{n}$.
\end{proof}

\subsection{Rational Cuspidal Subgroup} \label{rationalgrps}
As previously discussed, the cuspidal group $\calC_{1} (p)$ has three interesting rational subgroups, with the following clear inclusions: 

\begin{center}
    $ \calC_{1}^{\infty} (p) \subseteq \calC_{1}^{\Q} (p) \subseteq \calC_{1}(p) \left( \Q \right).$
\end{center}
Below we prove that the above inclusions are in fact equalities, and thus in this case we have a well defined \textit{rational cuspidal group} and an efficient method of computing it using the basis constructed in Theorem \ref{basis2}.

The first equality is a simple corollary of Theorem \ref{basis3}.
\begin{cor}
$\calC_{1}^{\infty} \left ( p \right) = \calC_{1}^{\mathbb{Q}} \left ( p \right).$
\end{cor}

\begin{proof}
 Take any  $ \left[ D \right] \in \calC_{1}^{\mathbb{Q}} \left ( p \right) \setminus \calC_{1}^{\infty} \left ( p \right) $, thus 
 \begin{center}
     $D = \displaystyle \sum_{i=0}^{n-1} d_{i}P_{i} + 
     d\displaystyle \sum_{i=0}^{n-1} Q_{i}$
\end{center}
for integers $d_{0}, \ldots, d_{n-1}, d$. Then $D \sim D - \divv (I_{n}^{d}) = D'$, and $ \left[ D' \right] \in \calC_{1}^{\infty} \left ( p \right)$.
 \end{proof}
 The second equality  is not as clear, but can still be proved using elementary methods. 
 \begin{prop} \label{c1inf}
     $\calC_{1}^{\Q}  ( p ) = \calC_{1} (p) \left( \mathbb{Q} \right). $
  \end{prop}

\begin{proof}
  By definition, there is a short exact sequence 
 \begin{center}
     $ 0 \longrightarrow \divv ( \mathcal{F}(p))  \longrightarrow \Div_{c}(p) \longrightarrow  \calC_{1} (p) \longrightarrow 0 $
 \end{center}
We observe that the above is in fact an exact sequence of $\Gal\left( \Bar{ \Q } / \Q \right)$-modules, and we take Galois cohomology: 
\begin{center}
   $ \Div_{c}(p)\left( \Q \right) \longrightarrow \calC_{1}\left( p\right) \left( \Q \right) \longrightarrow H^{1} \left( \text{Gal}  \left( \Bar{ \Q} / \Q  \right), \divv \left( \mathcal{F} \left( p\right) \right) \right)$
    \end{center}
    and our claim follows if we show $ H^{1} \left( \text{Gal}  \left( \Bar{ \Q} / \Q  \right), \text{div} \left( \mathcal{F}_{1} \left( p\right) \right) \right) = 0$.
This follows from the fact that the absolute Galois group acts through $\text{Gal} \left( \Q \left( \zeta_{p} \right)^{+} / \Q \right)$, which is a finite cyclic subgroup. Let $\sigma$ be a generator of the Galois group, and $N = 1 + \sigma + \ldots + \sigma^{n-1}$. As shown in \cite[Chapter 10, page 189]{guillot2018gentle}
our problem reduces to proving 
\begin{center}
$ \{ D \in  \divv\left( \mathcal{F}(p) \right) \ : ND =0 \}  = \left( 1 - \sigma \right) \divv\left( \mathcal{F}(p) \right) $
\end{center}
and this is clear since the Galois action fixes the $P_{i}$ and permutes the $Q_{i} $ cyclically.
\end{proof}

Computing the subgroup $\calC_{1}^{\infty} (p) = \calC_{1} (p) \left( \Q \right)$ is of particular interest since it's closely related to the rational torsion subgroup of $J_{1}(p)(\Q)_{\tors}$. More specifically, Ohta \cite{ohta2013eisenstein} proved that the two groups are equal up to $2-$torsion.

\section{The Cuspidal Subgroup of \texorpdfstring{$J_{1}(p)$}{c}} \label{examples1}
\subsection{Abstract Presentation of the Cuspidal Group}
Given our bases for $\mathcal{F}\left( p\right)$ and $\mathcal{F}^{\infty}(p)$ we can efficiently determine  the groups  $\mathcal{C}_{1}\left( p\right)$ and $\mathcal{C}_{1}\left( p\right) \left( \Q \right)$ abstractly. In  Tables \ref{CSX1} and  \ref{RCSX1} we compute these for all primes $11 \le p \le 97$. In fact, we computed these groups for all primes $11 \le p \le 997$, and a complete list along with the \texttt{Magma} code used for this computation can be found in the online repository:
\begin{center}
 \url{https://github.com/ElviraLupoian/ModularUnitsandCuspidalSubgroupsX1p/tree/main}.
 \end{center}

\begin{longtblr}[caption = {The cuspidal subgroup $\mathcal{C}_{1} \left( p\right)$}, 
label = {CSX1}]{ |c|c|}
\hline
$p$ & $\mathcal{C}_{1} \left( p\right)$ \\
\hline
$11$  & $ \Z / 25 \Z   $ \\ 
$13$ &  $\Z/19 \Z  \times  Z/19 \Z $ \\ 
$17$ & $ \Z / 292 \Z \times \Z / 1168 \Z $ \\ 
$19 $ &  $\Z / 1461 \Z \times  \Z / 13149 \Z $  \\
$23$ & $ \Z / 37181 \Z \times \Z / 4498901 \Z $\\
$29 $ & $ \Z / 4 \Z  \times  \Z / 4 \Z  \times  \Z / 4 \Z  \times  \Z / 4 \Z  \times  \Z / 9203892 \Z \times   \Z / 450990708 \Z  $ \\
       $ 31$ & $ \Z / 2 \Z  \times  \Z / 10 \Z  \times  \Z / 1772833370 \Z  \times  \Z / 8864166850 \Z $ \\
       $ 37 $ & $ \Z / 53505562232535 \Z \times  \Z / 481550060092815 \Z  $ \\
   $ 41$ &  $ \Z / 21553759881619888 \Z \times  \Z / 538843997040497200 \Z $ \\
       $ 43 $ & $ \Z/ 2 \Z \times  \Z / 2 \Z  \times  \Z / 223364647569268558 \Z \times  
        \Z/ 10944867730894159342 \Z $ \\
    $   47 $ & $ \Z / 142605986469692740469 \Z  \times  
        \Z / 75438566842467459708101 \Z $\\
    $ 53$ & $ \Z / 14032869452244904602299329 \Z \times 
        \Z / 2371554937429388877788586601 \Z  $ \\  
    $ 59 $ & $ \Z / 589324663207792234929168861989 \Z \times  
        \Z / 495622041757753269575431012932749 \Z  $ \\
$ 61 $ &   $ \Z / 77 \Z  \times   \Z / 77 \Z \times  
        \Z / 2249026400408198764708332679163 \Z $ \\  & $ \times  \Z / 56225660010204969117708316979075 \Z $ \\
$ 67 $ &  $  \Z / 661 \Z  \times  \Z / 661 \Z  \times  
        \Z / 20742411322218610498404968504426647  \Z $ \\ & $ \times  
        \Z / 2509831769988451870307001189035624287 \Z  $ \\
$  71 $  &  $  \Z / 701 \Z  \times  \Z / 701 \Z  \times  
        \Z / 24193505826034073099187101823196730189 \Z  $ \\ & $ \times  
        \Z / 29637044636891739546504199733415994481525 \Z  $ \\
 $73 $ &  $ \Z / 2 \Z  \times  \Z / 2 \Z  \times  \Z / 2 \Z  \times  \Z / 2 \Z \times  
        \Z / 3313439439643796256465574023292013345776614 \Z $ \\   & $ \times  
        \Z / 29820954956794166308190166209628120111989526 \Z  $ \\
 $79 $ &  $ \Z / 521 \Z  \times  \Z / 521 \Z  \times  
        \Z / 2263623089554573652699188302579475498580792443 \Z $ \\ & 
        $ \Z / 382552302134722947306162823135931359260153922867 \Z   $ \\
        $83 $ & $ \Z / 2406984131025101712234550549597592650903636598642273  \Z $ \\ & $ \times  \Z / 4046140324253195978266279473873553246169013122317660913 \Z  $ \\
       $ 89 $  &  $ \Z / 2 \Z  \times  \Z / 2 \Z  \times  \Z / 2 \Z \times  \Z / 2 \Z \times $ \\ & 
      $ \Z / 1522954443020854102271958820561189448280265949250848440630\Z  $ \\ &  $ \times  \Z / 184277487605523346374907017287903923241912179859352661316230 \Z $ \\
      $  97 $ &  $ \Z / 35 \Z  \times \Z / 35 \Z  \times $   \\ &  
       $ \Z / 46112087576831945308457230271075213082193874861568739344925034980 \Z $  \\ &  $\times  
        \Z / 737793401229311124935315684337203409315101997785099829518800559680 \Z  $ \\
\hline
\end{longtblr}

\begin{longtblr}[caption = {The rational cuspidal subgroup $\mathcal{C}_{1}^{\infty} \left( p\right)$}, 
label = {RCSX1}]{ |c|c|}
\hline
$p$ & $\mathcal{C}_{1} \left( p\right) \left( \Q \right)$ \\
\hline
$11$ & $ \Z / 5 \Z $\\
$13 $ & $ \Z / 19 \Z $ \\
$17$ & $\Z / 584 \Z $ \\
$ 19$ &   $\Z / 4383 \Z $ \\
$ 23$  &  $ \Z / 408991 \Z $ \\
$ 29 $ & $ \Z / 4\Z \times  \Z / 4 \Z  \times  \Z / 64427244\Z $\\
 $ 31 $  &  $ \Z / 10 \Z  \times  \Z / 1772833370 \Z $ \\
$ 37$  &  $ \Z / 160516686697605 \Z $ \\
 $ 41 $ &  $ \Z / 107768799408099440 \Z $\\
  $43$ &  $ \Z / 2 \Z  \times  \Z / 1563552532984879906 \Z $ \\
 $ 47 $ &   $ \Z /3279937688802933030787  $\\
$ 53 $ &  $ \Z / 182427302879183759829891277 \Z $\\
 $ 59 $ &  $ \Z / 17090415233025974812945896997681 \Z $ \\
 $ 61 $ &   $ \Z /77 \Z \times  \Z /11245132002040993823541663395815 \Z $ \\
$ 67 $ &    $ \Z /661 \Z \times   \Z / 228166524544404715482454653548693117 \Z $\\
 $ 71$  &   $ \Z /701 \Z \times  \Z / 846772703911192558471548563811885556615 \Z $\\
 $ 73 $ &  $ \Z / 2\Z  \times  \Z / 2\Z  \times   \Z /9940318318931388769396722069876040037329842 \Z $\\
 $ 79 $  & $ \Z / 521 \Z  \times  \Z / 29427100164209457485089447933533181481550301759 \Z  $\\
  $ 83 $ &   $ \Z /98686349372029170201616572533501298687049100544333193 \Z $\\ 
  $ 89 $ &  $ \Z /2 \Z \times  \Z / 2\Z  \times   \Z / 16752498873229395124991547026173083931082925441759332846930 \Z $ \\
 $ 97 $ &  $ \Z /35 \Z  \times \Z / 184448350307327781233828921084300852328775499446274957379700139920 \Z $\\
 \hline 
\end{longtblr}

\subsection{Explicit Bases of Divisors for $\mathcal{C}_{1}(p)$ and $\mathcal{C}_{1}(p)(\Q)$}
We our bases of modular units to explicitly investigate the structure of the cuspidal group and its rational subgroup. More specifically, for certain $p$ in the previously stated interval, we find independent cuspidal divisors $D_{i}$ and $D_{j}'$ such that
\begin{center}
$ \mathcal{C}_{1}(p) = \langle [D_{1} ] \rangle \oplus \ldots \oplus \langle [D_{l}] \rangle$ and $ \mathcal{C}_{1}(p)(\Q ) = \langle [D^{'}_{1} ] \rangle \oplus \ldots \oplus \langle [D^{'}_{k}] \rangle $.
\end{center}
In what follows, we write $D = P_{0} - Q_{0}$ and $D' = P_{0} - P_{n-1}$.  We found that for all $ 11 \le p \le 1000$, $[D]$ and $[D']$  generate the largest factor in the abstract decomposition of the cuspidal and rational cuspidal group respectively. We study the orders of these elements in detail in the next section. 

Below, we use these divisors to find good generators for the cuspidal group. Firstly, when $\mathcal{C}_{1}(p)$ and $\mathcal{C}_{1}(p)(\Q)$ have no small factors, linear combinations of $D$ and $D'$ are sufficient to generate our groups. Analysing our calculations, we note that there are two prominent structures. This is extremely likely to be a small prime phenomena, and we expect the general structure of the cuspidal group to be significantly more complicated as the level increases, however, we hope that by studying these simpler structures, we can gain a useful insight into the general structure of the cuspidal group. 
Firstly, suppose that $l=2$ and $k=1$. More precisely, if $p$ is in the set
\begin{align*}
    &p \in \{  13, 17,19, 23, 37,41, 47, 53,59, 83, 101, 107, 137, 149, 167, 173, 179, 191, 227, 263, 269, 293, 317, \\ & 347, 359, 383, 389, 419, 479, 503, 509, 557, 563, 569, 619, 653, 719,743,773,787,797,839,857,859,863, \\ & 887,907,983 \}.
    \end{align*}
    then $\mathcal{C}_{1}(p)(\Q)$ is cyclic and $\mathcal{C}_{1}(p) \cong \Z / n_{1} \Z \times \Z / n_{2} \Z$. We explicitly found 
\begin{align*}
    & \mathcal{C}_{1}(17) =   \langle [D' + 2D] \rangle \oplus \langle [D] \rangle,\ \ \ \mathcal{C}_{1}(17)(\Q) = \langle [D'] \rangle, \\
    & \mathcal{C}_{1}(19) = \langle [ D' + 3D \rangle \oplus \langle [D] \rangle, \ \ \ \mathcal{C}_{1}(19)(\Q) = \langle [D'] \rangle, \\
    & \mathcal{C}_{1}(37) = \langle [ D' + 3D] \rangle \oplus \langle [D] \rangle, \ \ \ \mathcal{C}_{1}(11)(\Q) = \langle [D'] \rangle.    \end{align*}
For all other named primes in our interval we found that 
\begin{center}
    $ \mathcal{C}_{1}(p) =\langle [nD' + n^2D] \rangle \oplus \langle [D] \rangle$ and  $ \mathcal{C}_{1}(p)(\Q) = \langle [D'] \rangle$, 
    \end{center}
where $n = d_{1}'/d_{1} \in \Z$.

There are $33$ primes in our interval with cuspidal group and rational cuspidal group of the form:
\begin{center}
$\mathcal{C}_{1}(p) \cong \Z/n_{1} \Z \times \Z/n_{1} \Z \times \Z /n_{2} \Z \times \Z / n_{3} \Z$ and $ \mathcal{C}_{1}(p)(\Q) \cong \Z / n_{1} \Z \times \Z /m \Z$, 
\end{center}
with $n = \frac{m}{n_{2}} \in \Z$. For most of these primes, namely if 
\begin{center}
    $ p \in \{ 43, 61, 67, 71, 79, 139, 223, 229, 283, 367, 439, 467, 499, 587, 607, 643, 727, 809, 823,947 \} $
\end{center}
we found that the order of $P_{0} -P_{1} $, which we denote by $c_{1}$, is divisible by $n_{1}$ and can take the following generators:
\begin{center}
  $ \mathcal{C}_{1}(p) = \langle \frac{c_{1}}{n_{1}}[P_{0} - P_{1}] \rangle \oplus \langle \frac{c_{1}}{n_{1}}[Q_{0} - Q_{1}] \rangle  \oplus \langle [ nD' +n^{2} D ] \rangle \oplus \langle [D] \rangle $  and $\mathcal{C}_{1}(p)(\Q) = \langle \frac{c_{1}}{n_{1}}[P_{0} - P_{1}] \rangle \oplus  \langle [D'] \rangle$.
\end{center} 
It is not clear if the rest of the primes in the interval there generators of the cuspidal group which  follow a similar pattern. We summarise our computation below. We write $c_{i}$ for the order of $P_{0} - P_{i}$. We give generators of the cuspidal group $D_{1}, D_{2}, D_{3}, D_{4}$ in Table \ref{RCSX7}, and note that for all such primes, the rational cuspidal group is generated by $D_{1}$ and $ D'$.
\begin{longtblr}[caption = {Generators of orders $(n_{1}, n_{1}, n_{2}, n_{3})$ }, 
label = {RCSX7}]{ |c|c|c| c|  c| c| c|}
\hline
$p$ & $n_{1}$ & $m/n_{2} $ &  $D_{1}$ & $D_{2}$ & $D_{3}$ & $D_{4}$ \\
\hline 
$ 97$  &$35 $ &  $ 4$ & $  \frac{c_{1}}{35} (P_0 - P_{1})$ &   $ \frac{c_{1}}{35} (Q_0 - Q_{1})$ & $-D' + 4D$ & $D$ \\
\hline
$ 139$ & $1385$ & $23$ & $\frac{c_{1}}{1385} (P_{0} - P_{1}) $ & $\frac{c_{1}}{1385} (Q_{0} - Q_{1})$ & $ 7D'  + 23D$ &  $D$ \\
\hline 
$193$ & $7$ & $8$ & $\frac{c_{2}}{7}(P_{0} - P_{2})$ & $\frac{c_{2}}{7} (Q_{0} - Q_{2})$ & $23D' + 24D$ & $D$ \\
\hline 
$ 307$ & $4$ & $51$ & $ \frac{c_{1}}{4} (P_{0} - P_{1})$ & $ \frac{c_{1}}{4} (Q_{0} - Q_{1})$ & $ 11D' + 51D$ & $D $ \\
\hline 
$449$ & $ 493$ & $56 $ & $ \frac{c_{1}}{493} (P_{0} - P_{1})$ & $ \frac{c_{1}}{493} (Q_{0} - Q_{1})$ & $ 13D' - 56D$ & $D $ \\
\hline 
$613$ & $ 33877$ & $51 $ & $ \frac{c_{1}}{33877} (P_{0} - P_{1})$ & $ \frac{c_{1}}{33877} (Q_{0} - Q_{1})$ & $ 2D' + 51D$ & $D$ \\
\hline 
$757$ & $ 1157087$ & $63 $ & $ \frac{c_{1}}{n_{1}} (P_{0} - P_{1})$ & $ \frac{c_{1}}{n_{1}} (Q_{0} - Q_{1})$ & $ -D' + 63D$ & $D $ \\
\hline 
$769$ & $ 409961464315$ & $32$ & $ \frac{c_{1}}{n_{1}} (P_{0} - P_{1})$ & $ \frac{c_{1}}{n_{1}} (Q_{0} - Q_{1})$ & $ -19D' + 32D$ & $D $ \\
\hline 
$829$ & $ 695$ & $69 $ & $ \frac{c_{1}}{695} (P_{0} - P_{1})$ & $ \frac{c_{1}}{695} (Q_{0} - Q_{1})$ & $ 31D' - 69D$ & $D$ \\
\hline 
$977$ & $69$ & $122$ &$ \frac{c_{1}}{69} (P_{0} - P_{1})$& $ \frac{c_{1}}{69} (Q_{0} - Q_{1})$ & $997D' + 854D$  & $D$ \\
\hline 
$991$ & $397$ & $165$ &$ \frac{c_{1}}{397} (P_{0} - P_{1})$& $ \frac{c_{1}}{397} (Q_{0} - Q_{1})$ & $59D' + 165D$  & $D$ \\
\hline 
\end{longtblr}
 If $p \in \{ 157, 877\}$ we found the following generators,
\begin{align*}
    & \mathcal{C}_{1}(157)= \langle [ P_0 -P_1], [Q_0 - Q_1], [13(P_0 - P_{n-1}) + 26(P_0 - Q_0)], [P_0 - Q_0]\rangle;\\
    &  \mathcal{C}_{1}(157)(\Q) = \langle [ P_{0}- P_{1}], [P_{0} - P_{n-1}] \rangle; \\
    & \mathcal{C}_{1}(877)= \langle [ P_0 -P_1], [Q_0 - Q_1], [-9(P_0 - P_{n-1}) + 73(P_0 - Q_0)], [P_0 - Q_0]\rangle; \\
    & \mathcal{C}_{1}(877)(\Q) = \langle [ P_{0}- P_{1}], [P_{0} - P_{n-1}] \rangle.
    \end{align*}
However, observe that these generators are not independent. In these examples we were unable to find independent generators with predictable (and small) coefficients.

\section{Orders of Cuspidal Divisors} \label{ordersofcusps}
In this section we give explicit formulae for the orders of the elements $D = P_{0} - Q_{0}$ and $D' = P_{0} - P_{n-1}$ in $J_{1}(p)$. Computing these involves working with certain \textit{circulant matrices} whose entries correspond to the orders of vanishing of the functions $E_{i}$ and $F_{j}$ at cusps. In the first subsection we give some general results on circulant matrices. In the final two subsections we apply these results to compute the orders of $D$ and $D'$.

\subsection{Circulant Matrices}
Let $k \in \Z_{>0}$ and $\mathbf{c} = (c_{1}, \ldots, c_{k}) \in \mathbb{C}^{k}$. The \textit{circulant matrix} defined by $\mathbf{c}$ is the $k \times k $ matrix
\begin{center}
    $\Circ(\mathbf{c}) := \begin{bmatrix}
        c_{1} & c_{2} & \ldots & c_{k-1} & c_{k} \\
        c_{2} & c_{3} & \ldots & c_{k}  & c_{1} \\
        \vdots & \vdots & \ddots & \vdots & \vdots \\
        c_{k} & c_{1} & \ldots & c_{k-2} & c_{k-1}
    \end{bmatrix}
    $
\end{center}
and the \textit{alternative circulant matrix} defined by $\mathbf{c}$ is the $k \times k $ matrix:
\begin{center}
    $ \Circa(\mathbf{c}) := \begin{bmatrix}
        c_{1} & c_{2} & \ldots & c_{k-1} & c_{k} \\
        c_{k} & c_{1} & \ldots & c_{k-2}  & c_{k-1} \\
        \vdots & \vdots & \ddots & \vdots & \vdots \\
        c_{2} & c_{3} & \ldots & c_{k} & c_{1}
    \end{bmatrix}.
    $
\end{center}
The two matrices are related by $\Circ(\mathbf{c}) = J \Circa( \mathbf{c})$, where $J$ is defined as 

\begin{center}
    $J = \Circ((1,0\ldots, 0)) = \begin{bmatrix}
        1 & 0 & \ldots & 0 & 0 \\
        0 & 0 & \ldots & 0 & 1 \\
        0 & 0 & \ldots & 1 & 0 \\
        \vdots & \vdots & \ddots & \vdots & \vdots \\
        0 & 1 & \ldots & 0 & 0
    \end{bmatrix}$.
\end{center}
The following result is proved by an elementary calculation. The author believes it to be well-known, but could not find a suitable reference and thus gives a short proof the reader's convenience. 
\begin{prop} \label{evals}
 Fix a primitive $k$th root of unity $\omega$. For $1 \le i \le k$, let $e_{i} = \displaystyle \sum_{j=1}^{k} (\omega^{i-1})^{j-1}c_{j} $ and $\mathbf{e}_{i} = ( 1, \omega^{i-1}, \ldots, (\omega^{i-1})^{k-1})^{T}$. Then $e_{i}$ is an eigenvalue for $\Circa(\mathbf{c})$ and $\mathbf{e}_{i}$ is an associated eigenvector. 
\end{prop}
\begin{proof}
 The $jth$ entry of $\Circa(\mathbf{c}) \cdot \mathbf{e}_{i}$ is 
 \begin{center}
    $ \displaystyle \sum_{s=1}^{j-1} c_{k-j+1 +s} (\omega^{i-1})^{s-1} + \displaystyle \sum_{s=j}^{k} c_{s+1-j} (\omega^{i-1})^{s-1} = \displaystyle \sum_{t=k+2-j}^{k} c_{t} (\omega^{i-1})^{t -k -2 +j} + \displaystyle \sum_{t=1}^{k+1-j} c_{t} (\omega^{i-1})^{j-1} = e_{i}(\omega^{i-1})^{j-1}$.  
 \end{center}
\end{proof}
From now on we write $J_{k,a}$ for the $k \times k $ matrix whose entries are all $a$. If $M$ is any $k \times k$ matrix, we write $\Bl(M,a)$ for the $2k \times 2k$ block matrix 
\begin{center}
    $\Bl(M,a)= \begin{bmatrix}
        M & J_{k,a} \\
        J_{k,a} & M 
    \end{bmatrix}$.
\end{center}
The following is an elementary calculation, whose proof is very similar to that of the previous result.
\begin{prop} \label{diagbl}
 Let $\mathbf{c}= (c_{1}, \ldots, c_{k}) \in \mathbb{C}^{k}$, $a \in \mathbb{C}$ and fix $\omega$ a primitive $k$th root of unity. 
 The pairs $(e_{i}, \mathbf{e}_{i})$ stated below are a complete eigensystem for the matrix $\Bl(\Circa(\mathbf{c}), a)$,
 \begin{itemize}
     \item $e = \displaystyle \sum_{t=1}^{k} c_{t} + ka $, $\mathbf{e} = (1, \ldots, 1)^{T}$;
     \item $e = \displaystyle \sum_{t=1}^{k} c_{t}  - ka $, $\mathbf{e} = (1, \ldots,1, -1, \ldots, -1)^{T}$;
     \item for $2 \le i \le k$, $e = \displaystyle \sum_{t=1}^{k} c_{t}(\omega^{i-1})^{t-1}$, $\mathbf{e} = (1, \omega^{i-1}, \ldots, (\omega^{i-1})^{k-1},1, \omega^{i-1}, \ldots, (\omega^{i-1})^{k-1})^{T}$;
     \item for $2 \le i \le k$, $e = \displaystyle \sum_{t=1}^{k} c_{t}(\omega^{i-1})^{t-1}$, $\mathbf{e} = (1, \omega^{i-1}, \ldots, (\omega^{i-1})^{k-1},0,0 , \ldots,0)^{T}$;   \end{itemize}
\end{prop}

\subsection{Orders of Rational Cuspidal Divisors}
The arguments appearing in this section are heavily inspired by those from \cite[Section 3]{yoo2023rational} and \cite[Section 2]{ling1997q}.

 Recall that any element in $\mathcal{C}_{1}(p)(\mathbb{Q})$ is represented by a divisor supported on $P_{0}, \ldots, P_{n-1}$. Define the following sets:
\begin{align*}
   & \mathcal{S}_{1} = \{ \displaystyle \prod_{i=0}^{n-1} E_{i} ( \tau )^{e_{i}} \  \vert  \ e_{i} \in \Q ,  \displaystyle \sum_{i=0}^{n-1} e_{i} =0 \}; \\  
    & \mathcal{S}_{2} = \{ \displaystyle \sum_{i=0}^{n-1} d{i} P_{i} ( \tau )^{e_{i}} \  \vert  \ d_{i} \in \Q ,  \displaystyle \sum_{i=0}^{n-1} d_{i} =0 \}. \\  
    \end{align*}
Let $\varphi$ be a homomorphism on $\mathcal{S}_{1}$, defined by 
\begin{center}
    $\varphi(\prod_{i=0}^{n-1} E_{i} ( \tau )^{e_{i}}) = \displaystyle \sum_{i=0}^{n-1} k_{i} P_{i}$
\end{center}
where $k_{i} = \displaystyle \sum_{j=0}e_{j}a_{[i+j]}$, with $[l+k] := l+k \mod{n}$ and $a_{i} = \frac{p}{2}B_{2}( \{ \frac{\alpha^{i}}{p} \} $ as before. First, we note that the image of $\varphi$ lies in $\mathcal{S}_{2}$, since 
\begin{center}
    $ \displaystyle \sum_{i=0}^{n-1} k_{i} = \left( \displaystyle \sum_{j=0}^{n=1} e_{j} \right) \times \left( \displaystyle \sum_{l=0}^{n=1} a_{l} \right) = 0 $. 
\end{center}

Moreover, this is in fact a bijection; and we argue this below. We begin by identifying both $\mathcal{S}_{1}$ and $\mathcal{S}_{2}$ with the following set
\begin{align*}
   & S = \{ ( k_{0}, k_{1}, \ldots, k_{n-1})^{T}  \vert  \ k_{i} \in \Q ,  \displaystyle \sum_{i=0}^{n-1} k_{i} =0 \}.
    \end{align*}
We simply view $\varphi$ as the function $S \rightarrow S$ given by the circulant matrix $M = \Circ(a_{0}, \ldots, a_{n-1})$, acting on the vector corresponding to the exponents. To prove that $\varphi$ is a bijection, we show that $M$ is invertible and explicitly compute its inverse. 

For the rest of this section we fix $\omega$, a primitive $n$th root of unity. For $0 \le i \le n-1$ we  write $\chi_{i}$ for the even Dirichlet character modulo $p$ satisfying $\chi_{i}(\alpha) = \omega^{i}$ and $B_{2,i} = B_{2,\chi_{i}} = p \sum_{a=1}^{p-1} \chi_{i}(a)B_{2}( \frac{a}{p})$ for the corresponding generalised Bernoulli number. Moreover, we write  
\begin{center}
    $b_{j} = \frac{4}{n} \displaystyle \sum_{i=0}^{n-1} (\omega^{ji} \cdot B_{2,i})^{-1}$
    \end{center}
for $0 \le j \le n-1$.
\begin{prop}
  The matrix $M$ is invertible, with inverse
  \begin{center}
    $ M^{-1}=  \Circ( b_{0}, b_{n-1},  \ldots , b_{1} )$.
    \end{center}
\end{prop}
\begin{proof}
We write $\mathbf{a} = (a_{0}, \ldots, a_{n-1})^{T}$. Recall that $\Circ(\mathbf{a}) = J\Circa(\mathbf{a})$ as in the previous subsection, and it follows from Proposition \ref{evals} that $\Circa(\mathbf{a}) = P^{-1} \Diag(y_{0}, \ldots, y_{n-1})P$ where:
\begin{center}
    $P = \begin{bmatrix}
    1  &  1 & \dots & 1 &  1  \\
         1 & \omega & \dots & \omega^{n-2} & \omega^{n-1}  \\
          
           \vdots & \vdots & \ddots & \vdots & \vdots \\
         1 & \omega^{n-1} & \dots & (\omega^{n-1})^{n-2} & ( \omega^{n-1})^{n-1}  \\            \end{bmatrix} $,
          $P^{-1} =  \frac{1}{n} \cdot \begin{bmatrix}
    1  &  1 & \dots & 1 &  1  \\
         1 & \omega^{-1} & \dots & \omega^{-n+2} & \omega^{1-n}  \\
          
           \vdots & \vdots & \ddots & \vdots & \vdots \\
         1 & \omega^{-n} & \dots & (\omega^{1-n})^{n-2} & ( \omega^{1-n})^{n-1}  \\            \end{bmatrix} $.           
         \end{center}
and $y_{i}$ are the eigenvalues, 
\begin{center}
    $y_{i} := \displaystyle \sum_{j=0}^{n-1} (\omega^{i})^{j} a_{j}  = \displaystyle  \frac{p}{2} \sum_{j=0}^{n-1} \chi_{i}( \alpha^{j}) B_{2}( \{ \frac{\alpha^{j}}{p} \}) = \frac{1}{4}B_{2,i}$.
    \end{center}
Note that $y_{i} \neq 0 $ for all $i$, since they are all factors of the cuspidal class number. The result then follows after a long, but elementary matrix multiplication.
\end{proof}

To compute the order of element in $\mathcal{C}_{1}(p)( \Q)$ we use the following result.
\begin{prop} 
\label{X1pcongs2}
Let $e_{i} \in  \Q$ and $
    f ( \tau) = \displaystyle \prod_{i=0}^{n-1} E_{i}(\tau)^{e_{i}}$.  Then $f \in \mathcal{F}^{\infty}(p)$  if and only if the exponents satisfy the following congruence conditions:
\begin{itemize}
\item[(i)] $ e_{i} \in \Z$ for all $i$;
    \item[(ii)] $ \displaystyle \sum_{i=0}^{n-1} e_{i} \alpha^{2i}  \equiv 0 \mod{p};$
    \item[(iii)] $ \displaystyle \sum_{i=0}^{n-1} e_{i} =0$.
\end{itemize}
\end{prop}
\begin{proof}
The reverse implication follows from Proposition \ref{X1pcongs}.  If such an $f$ is a modular unit and the exponents are integers, the second congruence holds by Proposition \ref{X1pcongs}. The last condition holds since the divisor of such a function has order of vanishing $\displaystyle \sum_{i=0}^{n-1} \frac{1}{12} e_{i}$ at all cusps $Q_{j}$, and this is assumed to be zero. Thus, it remains to show that the exponents are necessarily integers. For this we argue as in  \cite[Section 4]{ogg1974hyperelliptic}. More precisely, $f = (\Tilde{f})^{1/n}$ where $n \ge 1$ and $\Tilde{f}$ is a product of the $E_{i}$ with integer exponents. Then $\Tilde{f}$ is a power series in $q = e^{2\pi i \tau}$ with rational coefficients, with bounded denominators, and we conclude using Ogg's lemma from the same reference.   
\end{proof}

Let $D = \displaystyle \sum_{i=0}^{n-1} d_{i}P_{i} \in \Divv^{\infty}(p)$, then by the Manin-Drinfeld theorem there exists a minimal $N \in \Z \setminus\{0\}$, the order of $[D]$, such that 
\begin{center}
    $\divv(f) = N \displaystyle \sum_{i=0}^{n-1} d_{i}P_{i}$,
\end{center}
and since $N \sum_{i=0}^{n-1} d_{i}  =0$ and $\varphi$ is invertible, up to scaling, $f \in \mathcal{S}_{1}$, with integer exponents satisfying the congruences of Proposition \ref{X1pcongs2}. Note that this implication  also follows from the explicit basis computed in Theorem \ref{basis2}. Therefore, we have the following. 
\begin{prop}
For $D = \displaystyle \sum_{i=0}^{n-1} d_{i}P_{i} \in \Divv^{\infty}(p)$, the order of $[D] \in J_{1}(p)(\Q)_{\tors} $ is the smallest positive integer $N$ such that 
\begin{center}
    $M^{-1} \cdot ( Nd_{0}, \ldots, Nd_{n-1})^{T} $
\end{center}
satisfy the necessary and sufficient conditions of Proposition \ref{X1pcongs2}.
\end{prop}

We apply this proposition to $D' = P_{0} - P_{n-1}$. 
\begin{prop}
With notation as above, define 
\begin{itemize}
    \item $c_{0} = b_{0} - b_{1}$;
    \item $c_{1} = b_{n-1} - b_{0}$;
    \item for $2 \le i \le n-1$, $c_{i} = b_{n-i} - b_{n+1-i}$.
\end{itemize}
Then $c_{j} \in \Q$ for all $0 \le j \le n-1$. Let $L$ be the lowest common multiple of the denominators of $c_{j}$ and $T := \displaystyle \sum_{i=0}^{n-1} Lc_{i} \alpha^{2i}$. The order of $[ P_{0} - P_{n-1}] \in \mathcal{C}_{1}(p)(\Q)$ is $N = \frac{p}{\gcd(p,T)} L$.
\end{prop}
\begin{proof}
    Note that $(c_{0}, \ldots c_{n-1})^{T} = M^{-1}\cdot (1,0, \ldots, -1)^{T}$. The fact that $c_{j} \in \Q$ follows from the fact that $b_{s} \in \Q$ for all $s$ (since the eigenvalues $y_{i}$ and $y_{n-i}$ are complex conjugates; and so are $(\omega^{j})^{i}$ and $(\omega^{j})^{n-i}$). The rest follows from the previous proposition. 
\end{proof}

\subsection{The Order of $D$} The orders of arbitrary elements of $\mathcal{C}_{1}(p)$ is not as straightforward to compute. The main stumbling block is the fact that the analogous matrix to $M$ is $\Bl(\Circ(\mathbf{a}), \frac{1}{12})$, where $\mathbf{a} = ( a_{0}, \ldots, a_{n-1})^{T}$, which is singular since the sums of all rows (or columns) is the zero vector (cf Lemma \ref{columnsum}). Instead, we use Proposition \ref{diagbl} and the structure of the basis of $\mathcal{F}(p)$ to compute the order of $D= P_{0} -Q_{0}$ in $J_{1}(p)$.

We begin with the following simple lemma, which results from the construction of our basis.
\begin{lem}
Let $f(\tau) = \prod_{i=0}^{n-1}E_{i}(\tau)^{e_{i}} \prod_{j-1}^{n-1}F_{j}(\tau)^{f_{j}}$, where $e_{i}, f_{j} \in \Z$ satisfy the congruence conditions of Proposition \ref{X1pcongs}. Then
\begin{center}
    $f(\tau) = \displaystyle  \left( \prod_{i=0}^{n-2} G_{i}(\tau)^{e_{i}} \right) \left( \prod_{j=1}^{n-2} H_{j}(\tau)^{f_{j}} \right) G_{n-1}(\tau)^{C_{1}}H_{n-1}(\tau)^{C_{2}}$
\end{center}
where $C_{1}, C_{2} \in \Z$ are defined as
\begin{itemize}
    \item $C_{1} = \frac{1}{p} \left( e_{n-1} +  \displaystyle \sum_{i=0}^{n-2} \alpha^{2i+2} e_{i} \right);$
    \item $C_{2} = \frac{1}{12p} \left( f_{n-1} + C_{1} \beta p - \displaystyle \sum_{i=1}^{n-2} f_{i}( -\alpha^{2i+2} + p \beta ( \alpha^{2i+2} -1)) - \sum_{i=0}^{n-2} p e_{i}(\alpha^{2i+2} -1) \right).$
\end{itemize}
\end{lem}
\begin{proof}
 The fact that $C_{1}, C_{2} \in \Z$ follows from the congruence conditions of Proposition \ref{X1pcongs} and the rest follows from the definitions of the $G_{i}$ and $H_{j}$.    
\end{proof}
This gives us a strategy to compute the order of a cuspidal divisor in $J_{1}(p)$.  
\begin{cor}
 For any degree $0$ divisor  $\displaystyle \sum_{i=0}^{n-1} (x_{1} P_{i} + y_{i}Q_{i})$ supported on cusps, its order in $\mathcal{C}_{1}(p)$ is the smallest positive integer $N$ such that 
 \begin{itemize}
     \item[(i)]   $\Bl(\Circ(\mathbf{a}), \frac{1}{12}) \cdot (e_0, \ldots, e_{n-1}, 0, f_{1}, \ldots, f_{n-1})^{T} = (x_{0},\ldots, x_{n-1}, y_{0}, \ldots, y_{n-1})^{T}$;
     \item[(ii)] $Ne_{i}, Nf_{j} \in \Z$ for all i,j;
     \item[(iii)] $Ne_{i}, Nf_{j}$ satisfy the congruence conditions stated in Proposition \ref{X1pcongs}.
     \end{itemize}
\end{cor}
\begin{proof}
For any such divisor $\displaystyle \sum_{i=0}^{n-1} (x_{i} P_{i} + y_{i}Q_{i})$, its order in the Jacobian  is the smallest positive integer $N$ for which there exist integers $g_{i}, h_{j}$ with 
\begin{center}
   $ \divv \left( \displaystyle \prod_{i=0}^{n-1} G_{i}(\tau )^{g_{i}} \prod_{j=1}^{n-1} H_{j}(\tau )^{h_{j}} \right) = N \cdot \displaystyle \sum_{i=0}^{n-1} (x_{i}P_{i} + y_{i}Q_{i}).$ 
\end{center}
Thus, if one finds $\mathbf{x} = (e_{0}, \ldots, e_{n-1}, 0, f_{1}, \ldots, f_{n-1})^{T} \in \Q^{2n}$ with 
\begin{center}
     $\Bl(\Circ(\mathbf{a}), \frac{1}{12}) \cdot  \mathbf{x} = (x_{0},\ldots, x_{n-1}, y_{0}, \ldots, y_{n-1})^{T}$
     \end{center}
then the order is simply the smallest positive integer $N$ such that $N \cdot \mathbf{x} \in \Z^{2n}$ and it satisfies the congruence conditions of Proposition \ref{X1pcongs}. 
\end{proof}

We apply this to the divisor $P_{0} - Q_{0}$. In what follows, we assume as before that $\omega$ is a primitive $n$th root of unity. We want to find a solution $\mathbf{x} = (e_{0}, e_{1}, \ldots, e_{n-1}, f_{0}, \ldots, f_{n-1})^{T} \in \Q^{2n}$ with $f_{0} =0$ and such that 
\begin{center}
    $\Bl(\Circ(\mathbf{a}), \frac{1}{12}) \cdot  \mathbf{x} = (1,0, \ldots, 0, -1, 0, \ldots 0)^{T}.$
\end{center}
For $1 \le i \le n-1$ define
\begin{center}
    $\beta_{i} = \displaystyle \sum_{j=0}^{n-1} (\omega^i)^{j}a_{j} = \frac{1}{4}B_{2,i}$.
\end{center}
We diagonalise $\Bl(\Circa(\mathbf{a}), \frac{1}{12})$ using Proposition \ref{diagbl}. More precisely, $\Bl(\Circa(\mathbf{a}), \frac{1}{12}) = P^{-1} DP$ where 
 \begin{center}
  $P = \begin{bmatrix}
    1  &  1 & \dots & 1 &  1 & 1  &  1 & \dots & 1 &  1\\
         1 & \omega & \dots & \omega^{n-2} & \omega^{n-1}   &   1 & \omega & \dots & \omega^{n-2} & \omega^{n-1} \\
          
           \vdots & \vdots & \ddots & \vdots & \vdots  & \vdots & \vdots & \ddots & \vdots & \vdots \\
         1 & \omega^{n-1} & \dots & (\omega^{n-1})^{n-2} & ( \omega^{n-1})^{n-1}  & 1 & \omega^{n-1} & \dots & (\omega^{n-1})^{n-2} & ( \omega^{n-1})^{n-1} \\          
          1  &  1 & \dots & 1 &  1  & -1 & 0 & \ldots & 0 & 0 \\
         1 & \omega & \dots & \omega^{n-2} & \omega^{n-1} & -1 & 0 & \ldots & 0 & 0\\
          
           \vdots & \vdots & \ddots & \vdots & \vdots   & \vdots & \vdots & \ddots & \vdots & \vdots\\
         1 & \omega^{n-1} & \dots & (\omega^{n-1})^{n-2} & ( \omega^{n-1})^{n-1} & -1 & 0 & \ldots & 0 & 0\\         \end{bmatrix} $,
\end{center} 
and $D = \Diag( 0, \beta_{1}, \ldots, \beta_{n-1}, \frac{-n}{6}, \beta_{1}, \ldots, \beta_{n-1})$. Thus $\mathbf{x}$ defines the following system of equations:
\begin{itemize}
    \item $\beta_{j} \displaystyle \sum_{i=0}^{n-1} (\omega^{j})^{i} e_{i} =1$ for $1 \le j \le n-1;$
    \item $ \frac{-n}{6} \left( \displaystyle \sum_{i=0}^{n-1} e_{i} - f_{0} \right) = 1;$
    \item  $ \beta_{j} \displaystyle \sum_{i=0}^{n-1} (\omega^{j})^{i} ( e_{i} + f_{j}) = 1 - \omega^{j}$ for $1 \le j \le n-1$.
\end{itemize}
Setting $f_{0} = 0$, we find the unique solution to the above system:
\begin{itemize}
    \item $e_{i} := \frac{1}{n} \left( \frac{-6}{n} + \displaystyle \sum_{j=1}^{n-1} \frac{1}{\beta_{j}} ( \omega^{-i} )^{j} \right) = \frac{-6}{n^{2}} + b_{i}$ for $0 \le i \le n-1$;
    \item  $f_{i} = \frac{1}{n} \left( \displaystyle \sum_{j=1}^{n-1} \frac{\omega^{j}}{\beta_{j}} ( 1 - (\omega^{j})^{n-i}) \right) = b_{n-1} - b_{i-1}$ for $1 \le i \le n-1$;.
\end{itemize}
where the $b_{i}$ are  as in the previous subsection. 
Let $L$ be the lowest common multiple of the denominators of $e_{i}$ and $f_{j}$. Define 
\begin{itemize}
    \item $D_{1} = \displaystyle \sum_{i=0}^{n-1}Le_{i}\alpha^{2i}$;
  \item $D_{2} = \displaystyle \sum_{i=1}^{n-1}Lf_{i}\alpha^{2i}$;
  \item $D_{3} = \displaystyle \sum_{i=0}^{n-1}Lpe_{i} + \sum_{i=1}^{n-1} Lf_{i}$;
        \end{itemize}
        and $n_{1} = \gcd(D_{1}, p)$, $n_{2} = \gcd(D_{2}, p)$ and $n_{3} = \gcd(D_{3}, 12)$. We have proved the following. 
\begin{prop}
The order of $P_{0} -Q_{0}$ in $J_{1}(p)$ is   
 \begin{center}
   $N = \frac{p}{\gcd(n_{1},n_{2}) } \cdot \frac{12}{\gcd(n_{3},12)} \cdot L.$  
 \end{center}
\end{prop}

\bibliographystyle{amsplain} 
\bibliography{ref}

\end{document}